\documentclass[10pt]{amsart}

\pdfoutput=1

\usepackage{mathrsfs, amsmath, amsthm, amssymb,}%
\usepackage{tikz}%
\usepackage[linktocpage]{hyperref}%
\usepackage[all]{xy}%
\xyoption{2cell}%
\UseAllTwocells%

\hypersetup{
 pdfauthor={Daniel Schaeppi},
 pdfkeywords={Tannaka duality} {algebraic stacks} {Quasi-coherent sheaves} {weakly Tannakian categories}
}

\DeclareMathOperator{\id}{id}

\DeclareMathOperator{\op}{op}

\DeclareMathOperator{\Mod}{\mathbf{Mod}}

\DeclareMathOperator{\fp}{fp}

\DeclareMathOperator{\End}{End}
\DeclareMathOperator{\Spec}{Spec}

\DeclareMathOperator{\rk}{rk}
\DeclareMathOperator{\Sym}{Sym}
\DeclareMathOperator{\sgn}{sgn}
\DeclareMathOperator{\coev}{coev}
\DeclareMathOperator{\ev}{ev}

\DeclareMathOperator{\Ind}{Ind}

\DeclareMathOperator{\Aff}{\mathbf{Aff}}

\DeclareMathOperator{\QCoh}{\mathbf{QCoh}}

\DeclareMathOperator{\fpqc}{\mathit{fpqc}}

\DeclareMathOperator{\Aut}{Aut}

\DeclareMathOperator{\colim}{colim}

\DeclareMathOperator{\Cat}{\mathbf{Cat}}
\DeclareMathOperator{\Gpd}{\mathbf{Gpd}}

\DeclareMathOperator{\Set}{\mathbf{Set}}

\newcommand{\ca}[1]{\mathscr{#1}}

\newcommand{\ten}[1]{\mathop{{\otimes}_{#1}}}

\newcommand{\boxten}[1]{\mathop{{\boxtimes}_{#1}}}

\newcommand{\pb}[1]{\mathop{{\times}_{#1}}}

\newcommand{\defl}{\mathrel{\mathop:}=}

\theoremstyle{plain}
\newtheorem{thm}{Theorem}[section]
\newtheorem{prop}[thm]{Proposition}
\newtheorem{lemma}[thm]{Lemma}
\newtheorem{cor}[thm]{Corollary}

\theoremstyle{definition}
\newtheorem{example}[thm]{Example}
\newtheorem{rmk}[thm]{Remark}
\newtheorem{dfn}[thm]{Definition}
\newtheorem{notation}[thm]{Notation}

\newtheoremstyle{citing}{}{}{\itshape}{}{\bfseries}{.}{ }{\thmnote{#3}}
\theoremstyle{citing}

\newtheoremstyle{citingdfn}{}{}{}{}{\bfseries}{.}{ }{\thmnote{#3}}
\theoremstyle{citingdfn}

\numberwithin{equation}{section}


\author{Daniel Sch\"appi}
\title{Which abelian tensor categories are geometric?}

\address{Department of Mathematics\\
Middlesex College\\
The University of Western Ontario\\
1151 Richmond Street\\
London, Ontario\\
Canada, N6A 5B7}

\begin{document}

\begin{abstract}
 For a large class of geometric objects, the passage to categories of quasi-coherent sheaves provides an embedding in the 2-category of abelian tensor categories. The notion of weakly Tannakian categories introduced by the author gives a characterization of tensor categories in the image of this embedding.
  
 However, this notion requires additional structure to be present, namely a fiber functor. For the case of classical Tannakian categories in characteristic zero, Deligne has found intrinsic properties --- expressible entirely within the language of tensor categories --- which are necessary and sufficient for the existence of a fiber functor. In this paper we generalize Deligne's result to weakly Tannakian categories in characteristic zero. The class of geometric objects whose tensor categories of quasi-coherent sheaves can be recognized in this manner includes both the gerbes arising in classical Tannaka duality and more classical geometric objects such as projective varieties over a field of characteristic zero.

 Our proof uses a different perspective on fiber functors, which we formalize through the notion of geometric tensor categories. A second application of this perspective allows us to describe categories of quasi-coherent sheaves on fiber products.
\end{abstract}

\maketitle

\tableofcontents

\section{Introduction}\label{section:introduction}

 In \cite{LURIE}, Lurie gave a characterization of the tensor functors between categories of quasi-coherent sheaves
\[
 \QCoh(Y) \rightarrow \QCoh(X)
\]
 on stacks $X$ and $Y$ which are induced by a morphism $f \colon X \rightarrow Y$. He called such tensor functors \emph{tame}. Brandenburg and Chirvasitu showed that if $Y$ is a quasi-compact quasi-separated scheme, then \emph{every} left adjoint tensor functor between the categories of quasi-coherent sheaves is tame (see \cite{BRANDENBURG_CHIRVASITU}). In a similar vein, the author showed that the same is true for categories of quasi-coherent sheaves on \emph{Adams stacks}. These are stacks on the $\fpqc$-site\footnote{Since this site is a large category, there are some set-theoretical difficulties when defining arbitrary sheaves and stacks on it. These can be circumvented in several ways, for example by using Grothendieck universes.} $\Aff_R$ associated to flat affine groupoids which also satisfy the strong resolution property: the quasi-coherent sheaves with duals form a generator of the category of all quasi-coherent sheaves. The gerbes associated to Tannakian categories are examples of such stacks, as are more classical geometric objects such as projective varieties and quasi-compact semi-separated schemes with the strong resolution property.

 Thus, we can embed a large class of geometric objects in the 2-category of tensor categories. This leads to a wealth of natural questions: what do various geometric constructions and properties correspond to in the world of tensor categories? Note that many questions of this kind have been studied before Lurie's result was known, in order to generalize classical concepts to noncommutative algebraic geometry.

 In \cite{SCHAEPPI_TENSOR}, the author has shown that products of Adams stacks correspond to the Kelly tensor product of categories (a tensor product which generalizes Deligne's tensor product of abelian categories, in the sense that the two coincide whenever the latter exists, see \cite{IGNACIO_TENSOR}). In \S \ref{section:fiber} we use similar ideas to the ones used in \cite{SCHAEPPI_TENSOR} to extend this result to fiber products of Adams stacks.

 \begin{thm}\label{thm:fiber_product}
  Let $X$, $Y$, and $Z$ be Adams stacks. Then there is an equivalence
\[
 \QCoh_{\fp}(X \pb{Z} Y) \simeq \QCoh_{\fp}(X) \boxten{\QCoh_{\fp}(Z)} \QCoh_{\fp}(Y) 
\]
 of symmetric monoidal $R$-linear categories.
 \end{thm}

 Here $\QCoh_{\fp}$ denotes the full subcategory of finitely presentable quasi-coherent sheaves. The construction $\ca{A} \boxten{\ca{C}} \ca{B}$ is a generalization of a tensor product for finite tensor categories introduced by Greenough in \cite[Definition~3.5]{GREENOUGH} (see \S \ref{section:fiber} for a precise definition).

 These results about categories of quasi-coherent sheaves rely on an answer to the following more basic question: can we give a characterization of tensor categories which arise from algebro-geometric objects? 

 The notion of weakly Tannakian category introduced by the author (see \cite{SCHAEPPI_STACKS, SCHAEPPI_TENSOR}) gives such a characterization, where the geometric objects in question are Adams stacks over a fixed commutative ring $R$. Recall that a category is \emph{ind-abelian} if its category of ind-objects is abelian. There is also an intrinsic description of such categories (see \cite{SCHAEPPI_TENSOR}). Note that any abelian category is in particular ind-abelian. We call a symmetric monoidal category with finite colimits \emph{right exact} if tensoring with any fixed object gives a right exact functor.

\begin{dfn}\label{dfn:weakly_tannakian}
 Let $\ca{A}$ be an ind-abelian right exact symmetric monoidal $R$-linear category, and $B$ a commutative $R$-algebra. A functor
\[
 w \colon \ca{A} \rightarrow \Mod_B
\]
 is called a \emph{fiber functor} if it is faithful, flat, and right exact. 

 We say that $\ca{A}$ is \emph{weakly Tannakian} if it satisfies the conditions:
\begin{enumerate}
 \item[(i)] There exists a fiber functor $w \colon \ca{A} \rightarrow \Mod_B$ for some commutative $R$-algebra $B$;
\item[(ii)] For all objects $A \in \ca{A}$ there exists an epimorphism $A^{\prime} \rightarrow A$ such that $A^{\prime}$ has a dual.
\end{enumerate}
\end{dfn}

 See for example \cite[\S 3.1]{SCHAEPPI_TENSOR} for a definition of flat functors. If $\ca{A}$ is abelian, then a functor $w$ as above is flat if and only if it is left exact.
 
 One of the central results of \cite{SCHAEPPI_TENSOR} is that $\ca{A}$ is weakly Tannakian if and only if it is equivalent to the category $\QCoh_{\fp}(X)$ of finitely presentable quasi-coherent sheaves on an Adams stack $X$.

 Note, however, that this characterization relies not just on properties of the symmetric monoidal category $\ca{A}$, but also on the existence of additional structure (namely the fiber functor). The problem of finding intrinsic axioms for the existence of fiber functors was called the \emph{description problem} in \cite{WEDHORN} (as opposed to the \emph{recognition problem} for categories of quasi-coherent sheaves on Adams stacks, for which the notion of weakly Tannakian already provides a satisfactory answer).

 In \cite[\S 7]{DELIGNE}, Deligne solved the description problem for Tannakian categories over a field of characteristic zero. Recall that an abelian weakly Tannakian category over a field $k$ is called \emph{Tannakian} if there exists a fiber functor landing in $K$-vector spaces for some field extension $k \subseteq K$, if \emph{every} object has a dual, and the endomorphism ring of the unit object is equal to $k$. The Tannakian categories are precisely the categories of coherent sheaves on $\fpqc$-gerbes with affine band over $k$ (see \cite[Th\'eor\`eme~1.12 and \S 3.6]{DELIGNE}).

 Deligne uses exterior powers and traces of endomorphisms --- notions which make sense in any right exact symmetric monoidal category in characteristic zero --- in order to prove the following description result (see \cite[Th\'eor\`eme~7.1]{DELIGNE}). 

\begin{thm}[Deligne]\label{thm:deligne_description}
 Let $k$ be a field of characteristic zero. Let $\ca{A}$ be an abelian symmetric monoidal $k$-linear category such that the $k$-algebra of endomorphisms of the unit object is equal to $k$. Then $\ca{A}$ is Tannakian if and only if for every object $X \in \ca{A}$, there exists an integer $n \geq 0$ such that $\Lambda^n X=0$.
\end{thm}

 We recall the definitions of exterior powers and traces in \S \ref{section:description}. For the moment it suffices to know that they coincide with the classical constructions if the symmetric monoidal category in question is the category of $B$-modules of a commutative $R$-algebra $B$, and that they are preserved by tensor functors.

 In \S \ref{section:description}, we will generalize Deligne's result in order to solve the description problem for weakly Tannakian $R$-linear categories if $R$ is a $\mathbb{Q}$-algebra. To do this we need to identify a few additional axioms which are either obvious or provable in the Tannakian case. 

 Let $\ca{A}$ be a weakly Tannakian category. Recall that the \emph{rank} of an object with dual is the trace of the identity morphism. Deligne uses the fact that in any category satisfying the conditions of Theorem~\ref{thm:deligne_description}, an object has zero rank if and only if it is zero. His proof does not generalize to the weakly Tannakian case, but it is clearly a necessary condition for the existence of a fiber functor: any tensor functor preserves rank, and a finitely generated projective module with rank zero must be equal to zero. This identifies one of the additional axioms that we require.

 The second axiom concerns epimorphisms $ p \colon X \rightarrow Y$ between objects with duals in $\ca{A}$. The fiber functor sends such an epimorphism to an epimorphism between finitely generated projective $B$-modules, hence to a split epimorphism. It follows that the kernel $K$ of $p$ has a dual as well, and that the dual sequence
\[
 \xymatrix{0 \ar[r] & Y^{\vee} \ar[r]^{p^{\vee}} & X^{\vee} \ar[r] & K^{\vee} \ar[r] & 0 }
\]
 is exact. This is in particular true if the target $Y$ is the unit object $I \in \ca{A}$. Note that this property clearly holds under the assumptions of Theorem~\ref{thm:deligne_description} since every object has a dual. For a general weakly Tannakian category this need not be the case. 

 In \S \ref{section:description}, we will see that the properties we identified are enough to ensure the existence of a fiber functor.

\begin{thm}\label{thm:description}
 Assume that $R$ is a  $\mathbb{Q}$-algebra. Let $\ca{A}$ be an ind-abelian $R$-linear right exact symmetric monoidal category. Suppose that $\ca{A}$ satisfies the following four conditions:
\begin{itemize}
 \item[(i)] For every object $X \in \ca{A}$, there exists an epimorphism $X^{\prime} \rightarrow X$ such that $X^{\prime}$ has a dual.

\item[(ii)] If $X \in \ca{A} $ has a dual and if $\rk(X)=0$, then $X=0$.

\item[(iii)] If $X \in \ca{A}$ has a dual, then there exists an integer $n \geq 0$ such that $\Lambda^n X=0$.
 
 \item[(iv)] If $X \in \ca{A}$ has a dual and $p \colon X \rightarrow I$ is an epimorphism, then the dual morphism $I \cong I^{\vee} \rightarrow X^{\vee}$ is a monomorphism whose cokernel has a dual. 
\end{itemize}
 Then $\ca{A}$ is weakly Tannakian. In particular, there exists an Adams stack $Y$ such that $\ca{A} \simeq \QCoh_{\fp}(Y)$.
\end{thm}
 
 Both our result about quasi-coherent sheaves on fiber products (Theorem~\ref{thm:fiber_product}) and the above description result rely on a slightly different perspective on fiber functors involving the category $\Ind(\ca{A})$ of ind-objects in $\ca{A}$. We call categories of this form \emph{geometric tensor categories}, and we develop their basic theory in \S \ref{section:geometric}. We use this new perspective in \S \ref{section:affine} to study functors corresponding to affine morphisms between Adams stacks.

 We give a proof of Theorem~\ref{thm:fiber_product} in \S \ref{section:fiber}, and we prove our description theorem (Theorem~\ref{thm:description}) in \S \ref{section:description}.

 We end with two applications of our results to illustrate how they can be applied to the study of the duality between stacks and tensor categories. In \S \ref{section:etale} we use Theorem~\ref{thm:fiber_product} to give a description of finite \'etale morphisms between Adams stacks in terms of projective separable algebras in the category of quasi-coherent sheaves. As one of the referees pointed out, this can also be proved more directly (and for a broader class of stacks) using different methods. As such it should be seen as a proof of concept that results along this line \emph{can} be proved using the duality between stacks and tensor categories.
  
 The second result on the other hand is a genuine application of the techniques developed in this paper. In \S \ref{section:complete} we use the above description theorem to show that the 2-category of Adams stacks over a commutative ring $R$ containing $\mathbb{Q}$ is bicategorically complete.

 Throughout, we fix a commutative ring $R$. We write $\ca{AS}$ for the 2-category of Adams stacks over $R$, and $\ca{RM}$ for the 2-category of right exact symmetric monoidal $R$-linear categories (that is, finitely cocomplete symmetric monoidal $R$-linear categories with the property that tensoring with any fixed object gives a right exact functor), and right exact symmetric strong monoidal functors between them.
 
\section*{Acknowledgments}

 I want to thank Martin Brandenburg for pointing out several corrections and for suggesting some improvements of the exposition. The results about pullbacks along affine maps (Corollary~\ref{cor:pushout_along_affine_functor}) as well as the description of such (Proposition~\ref{prop:pushout_along_affine}) were independently discovered by him, and will appear in his PhD thesis.
 
  I thank both referees for carefully reading the manuscript. Their comments improved the readability in several places and simplified some of the proofs.
\section{Geometric and pre-geometric tensor categories}\label{section:geometric}

 In this section we introduce the notions of geometric and pre-geometric tensor categories. Most of the material presented here is already present in \cite{SCHAEPPI_STACKS} and \cite{SCHAEPPI_TENSOR}, though the focus will be on the category of ind-objects (rather than its subcategory of finitely presentable objects).

 Unless specified otherwise, all categories and functors we consider are $R$-linear. We first review some categorical terminology and fix some notation.

\subsection{Categorical background and notation}
 
 Throughout, we denote the hom-set of a category $\ca{C}$ between two objects $A$ and $B$ by $\ca{C}(A,B)$. We mostly deal with $R$-linear categories, which means that the hom-sets are endowed with $R$-module structures, and that composition is $R$-bilinear.
 
 Left adjoint functors are usually denoted by $F$, right adjoints by $U$ (to remind the reader of the free-forgetful adjunction in algebraic contexts). We sometimes write $F \dashv U$ to indicate that $F$ is left adjoint to $U$. The unit and counit (that is, the morphism corresponding to the respective identity morphisms under the adjunction) are always denoted by $\eta_A \colon A \rightarrow UF(A)$ and $\varepsilon_A \colon FU(A) \rightarrow A$. A functor which preserves all small limits (respectively all small colimits) is called \emph{continuous} (respectively \emph{cocontinuous}). In particular, a left adjoint is always cocontinuous. A category which has all small limits (colimits) is called \emph{complete} (\emph{cocomplete}).

 In a large class of categories, cocontinuity of a functor is equivalent to the existence of a right adjoint (results along those lines are known as adjoint functor theorems). This is in particular the case if both categories are \emph{locally finitely presentable}, which will be the case for the categories we consider. Recall that an object $C$ of a category $\ca{C}$ is called \emph{finitely presentable} if the representable functor $\ca{C}(C,-) \colon \ca{C} \rightarrow \Set$ preserves filtered colimits. The category $\ca{C}$ is \emph{locally finitely presentable} if there exists a strong generating sets of finitely presentable objects. A category that is both locally finitely presentable and abelian is in particular Grothendieck abelian.

 A category $\ca{C}$ is called \emph{symmetric monoidal} if it is equipped with a \emph{tensor product} functor $\ca{C} \times \ca{C} \rightarrow \ca{C}$ (denoted by $-\otimes -$), a unit object $I \in \ca{C}$, and coherent natural isomorphisms $(A\otimes B) \otimes C \cong A \otimes (B\otimes C)$, $A \otimes I \cong A \cong I \otimes A$, and a symmetry isomorphism $s_{A,B} \colon A \otimes B \cong B\otimes A$. These data are subject to various axioms, for example a pentagon diagram that relates the two ways of moving between the various bracketings of four objects.

 A functor $F \colon \ca{C} \rightarrow \ca{C}^{\prime}$ between symmetric monoidal categories is called a \emph{symmetric (lax) monoidal functor} if it is equipped with a morphism $\varphi_0 \colon I \rightarrow FI$ and natural morphisms $\varphi_{A,B} \colon FA \otimes FB \rightarrow F(A\otimes B)$, subject to compatibility conditions with the structure morphisms of a symmetric monoidal category. If $\varphi_0$ and all the $\varphi_{A,B}$ are invertible, the functor $F$ is called \emph{symmetric strong monoidal}. If we want to emphasize the dependence of the structure morphisms on the functor $F$, we will write $\varphi_0^F$ and $\varphi^F_{A,B}$ instead of $\varphi_0$ and $\varphi_{A,B}$.

 When a category is both symmetric monoidal and locally finitely presentable, we also require that the two structures are compatible in the following sense (cf.\ \cite{KELLY_FINLIM}). A \emph{locally finitely presentable symmetric monoidal category} is a category $\ca{C}$ that is both locally finitely presentable and symmetric monoidal, the unit object is finitely presentable, and finitely presentable objects are closed under the tensor product.
 
 Similarly, when a category is both $R$-linear and symmetric monoidal, we demand that these two structures are compatible, that is, we demand that the tensor product functor is $R$-bilinear.

 When dealing with stacks and categories equipped with various algebraic structures, we are inevitably lead into the world of two-dimensional category theory: in both cases there is a natural notion of ``morphism between morphisms.'' A \emph{(strict) 2-category} $\ca{K}$ has for each pair of objects $A$, $B$ a \emph{category} $\ca{K}(A,B)$ (rather than a mere set), and composition is given by a \emph{functor} $\ca{K}(B,C) \times \ca{K}(A,B) \rightarrow \ca{K}(A,C)$ (subject to the usual associativity and identity conditions). The morphisms in a 2-category are called \emph{1-cells}, and the morphisms between 1-cells are the \emph{2-cells}. When this terminology is used, the objects are often referred to as \emph{0-cells}.  The primordial example of a strict 2-category is the 2-category of small categories, with 1-cells the functors and 2-cells the natural transformations.

 The notion of a strict 2-category can be weakened to obtain the notion of a \emph{bicategory}, where the axioms only hold up to coherent isomorphisms. While we have no need of this level of generality, the notion of limit and colimit in 2-categories we will consider is the weak one. For example, strict colimits in a 2-category are defined from limits in the category of small categories via a natural \emph{isomorphism}
\[
 \ca{K}(\colim D_i,A) \cong  \lim \ca{K}(D_i,A)
\]
 of categories, subject to naturality conditions. A \emph{bicategorical colimit} or \emph{bicolimit} is instead defined via an \emph{equivalence}
\[
 \ca{K}(\colim D_i,A) \simeq  \lim \ca{K}(D_i,A)
\]
 of categories. These are only unique up to equivalence (not up to isomorphism). Moreover, rather than considering universal properties with strictly commuting diagrams, we will mostly be interested in diagrams that commute up to coherent isomorphisms, and objects which are universal among such. The main example we consider are \emph{bicategorical pushouts} (respectively pullbacks), which are universal among squares which commute up to a specified isomorphism (not subject to any further conditions).

% A general theme of higher category is that often the weak notions can be computed in terms of the strict ones. An example of this is the following. Given two morphisms $f\colon A \rightarrow B$ and $g \colon C \rightarrow B$ in a 2-category $\ca{K}$, one of which is an \emph{isofibration}, then the strict pullback of $f$ and $g$ has the universal property of a bicategorical pullback of $f$ and $g$. Here a morphism in a 2-category is called an isofibration if 

 The general theory of such colimits ``up to coherent isomorphism'' can be made precise using the notion of \emph{weighted} colimits, which is discussed in detail in \cite{KELLY_BASIC} (under the name of indexed colimit). The relevant bicategorical version can be found in \cite{STREET_FIBRATIONS, STREET_CORRECTIONS}. For most results in this paper no knowledge of general weighted colimits is required.

\subsection{Definitions and basic properties}

\begin{dfn}\label{def:tensor_category}
 Let $\ca{C}$ be a locally finitely presentable symmetric monoidal $R$-linear category. If $\ca{C}$ is closed, we call it a \emph{tensor category}. As mentioned above, the unit object of a tensor category is thus always assumed to be finitely presentable. A functor
 \[
 F \colon \ca{C} \rightarrow \ca{D}
 \]
 between two tensor categories is called a \emph{tensor functor} if it is symmetric strong monoidal and if it has a right adjoint (equivalently, if it preserves colimits). We write $\ca{T}$ for the 2-category of tensor categories, tensor functors, and symmetric monoidal natural transformations between them.
\end{dfn}

\begin{rmk}
 Note that there are various definitions of tensor categories in the literature (as opposed to the term monoidal category, which is unambiguous). For us, a tensor category will in particular always be cocomplete and closed. 
 % Moreover, we insist that the monoidal structure interacts well with the property of being locally finitely presentable. More precisely, we ask that the unit object $I$ is finitely presentable, and that finitely presentable objects are closed under the tensor product.
\end{rmk}

\begin{dfn}\label{def:pregeometric_category}
 A tensor category $\ca{C}$ is called \emph{pre-geometric} if it is abelian and the objects with duals form a generator of $\ca{C}$. In particular, $\ca{C}$ is a Grothendieck abelian category.
\end{dfn}

 Recall that a set $\ca{G}$ of objects in a category $\ca{C}$ is called \emph{dense generator} if every object $M$ of $\ca{C}$ is the colimit of the following canonical diagram associated to $M$. The objects of the indexing category are the morphisms with target $M$ whose domain lies in $\ca{G}$. The morphisms of the indexing category are the morphisms in $\ca{C}$ that make the evident triangle commutative. Finally, the diagram in $\ca{C}$ is obtained by sending a morphism (that is, an object in this indexing category) to its domain.

\begin{rmk}\label{rmk:duals_dense}
 Day and Street showed that a set of objects forms a generator of a Grothendieck category $\ca{C}$ if and only if it is a dense generator of $\ca{C}$ (see \cite[Theorem~2 and Example~(3)]{DAY_STREET_GENERATORS}).  
\end{rmk}

\begin{rmk}\label{rmk:flat_resolutions_hence_tame}
 In any pre-geometric category, there exist enough flat resolutions to set up a good theory of Tor functors. Thus pre-geometric categories are tame in the sense of Lurie (see \cite[Remark~5.3]{LURIE} and the proof of Lemma~\ref{lemma:flatness_from_filtration}).
\end{rmk}

 Recall that $\ca{RM}$ denotes the 2-category of right exact symmetric monoidal categories and right exact symmetric strong monoidal functors between them. The passage to ind-objects (respectively the restriction to finitely presentable objects) sets up a close relationship between 2-categories $\ca{RM}$ and $\ca{T}$. Note, however, that restriction to finitely presentable objects is \emph{not} functorial, since there is no requirement that tensor functors need to preserve such. This problem does not arise in the case of pre-geometric categories.

\begin{lemma}\label{lemma:pregeometric_implies_finitary}
 Let $\ca{D}$ be a pre-geometric category. Then every tensor functor with domain $\ca{D}$ sends finitely presentable objects to finitely presentable objects. In particular, there is an equivalence
\[
 \ca{RM}(\ca{D}_{\fp},\ca{A}) \simeq \ca{T}(\ca{D},\Ind(\ca{A})) 
\]
 of categories which is natural in the right exact symmetric monoidal category $\ca{A}$.
\end{lemma}

\begin{proof}
 Since the objects with duals form a dense generator of the Grothendieck abelian category $\ca{D}$ (see Remark~\ref{rmk:duals_dense}), every finitely presentable object can be written as cokernel of a morphism between objects with duals. The claim follows from the fact that tensor functors preserve duals. 
\end{proof}

\begin{dfn}
 Let $\ca{C}$ be a tensor category over $R$ and let $B$ be a commutative $R$-algebra. A tensor functor $w\colon \ca{C} \rightarrow \Mod_B$ is called a \emph{faithfully flat covering} if $w$ is faithful and exact. 

 A pre-geometric tensor category $\ca{C}$ is called \emph{geometric} if there exists a commutative $R$-algebra $B$ and a faithfully flat covering $\ca{C}\rightarrow \Mod_B$.
\end{dfn}

 Note that this does not cover all the examples of abelian tensor categories arising from algebraic geometry. For example, categories of quasi-coherent sheaves on schemes which do not satisfy the resolution property are not geometric in this sense. From this point of view it would make sense to reserve the unqualified adjective ``geometric'' for the broader class. Since there are currently no recognition results for these more general tensor categories, we have instead decided to include the resolution property in the definition of geometric tensor categories.
 
 The question of whether or not a given scheme has the resolution property is generally hard to answer. However, a large class of schemes (for example smooth separated schemes and projective varieties) do have the resolution property.

\begin{prop}\label{prop:geometric_equals_weakly_Tannakian}
 An essentially small category $\ca{A}$ is weakly Tannakian if and only if $\Ind(\ca{A})$ is geometric. A tensor category $\ca{C}$ is geometric if and only if the full subcategory $\ca{C}_{\fp}$ of finitely presentable objects is weakly Tannakian. In fact, these constructions define mutually inverse 2-equivalences between the 2-category of weakly Tannakian categories and the 2-category of geometric categories.
\end{prop}

\begin{proof}
 A symmetric monoidal structure on $\ca{A}$ with the property that tensoring with a fixed object is right exact induces a symmetric monoidal closed structure on $\Ind(\ca{A})$ (this follows for example from work of Day \cite{DAY_REFLECTION}, see \cite[Proposition~3.2.3]{SCHAEPPI_STACKS} for details). Moreover, from Part~(ii) of the definition of weakly Tannakian categories it follows that $\Ind(\ca{A})$ is pre-geometric. 
 
 From Lemma~\ref{lemma:pregeometric_implies_finitary} we know that all functors between geometric categories preserve finitely presentable objects. This shows that the assignment which sends a weakly Tannakian category $\ca{A}$ to $\Ind(\ca{A})$ is essentially surjective on 1-cells. From density of the objects with duals it follows that the assignment is full and faithful on 2-cells.
 
 Finally, to see that $\Ind(\ca{A})$ is geometric if and only if $\ca{A}$ is weakly Tannakian, note that left Kan extension along the inclusion $\ca{A}\rightarrow \Ind(\ca{A})$ induces an equivalence between fiber functors on $\ca{A}$ and affine coverings of $\Ind(\ca{A})$, with inverse equivalence given by restriction along this inclusion.
\end{proof}

\begin{thm}\label{thm:geometric_equals_qc}
 A tensor category $\ca{C}$ is geometric if and only if there exists an Adams stack $X$ and an equivalence
 \[
  \ca{C} \simeq \QCoh(X)
 \]
 of tensor categories. The pseudofunctor which sends an Adams stack $X$ to the category $\QCoh(X)$ and a morphism of stacks to the inverse image functor is a biequivalence between the 2-category of Adams stacks and the full subcategory of $\ca{T}$ consisting of geometric categories.
\end{thm}

\begin{proof}
 The first two assertions follow immediately from Proposition~\ref{prop:geometric_equals_weakly_Tannakian}, the recognition theorem (see \cite[Theorem~1.6]{SCHAEPPI_TENSOR}), and the embedding theorem for weakly Tannakian categories (see \cite[Theoren~1.3.3]{SCHAEPPI_STACKS}).
 
 As one of the referees pointed out, checking that $\QCoh(X)$ is a locally finitely presentable as a symmetric monoidal category is slightly subtle. This follows for example from the two facts that $\QCoh(X)$ is equivalent to the category of comodules of any flat Hopf algebroid that provides a presentation for $X$ (see e.g.\ \cite[\S 3.4]{NAUMANN} and \cite[Remark~2.39]{GOERSS}), and that a comodule is finitely presentable if and only if its underlying module is (see \cite[Proposition~1.3.3]{HOVEY}). From this we deduce that the unit object is finitely presentable, and that finitely presentable objects are closed under tensor products. Since the objects with duals are finitely presentable and form a (dense) generator it follows that $\QCoh(X)$ is indeed locally finitely presentable as a symmetric monoidal category.
\end{proof}

 Proposition~\ref{prop:geometric_equals_weakly_Tannakian} shows that geometric categories are equivalent to weakly Tannakian categories. In other words, we have so far simply introduced a slightly different language to talk about the same mathematical objects. However, this language is better suited for our purposes since it is closer to the geometric interpretation and it allows us to talk about quasi-coherent sheaves directly.

\section{Affine and cohomologically affine functors}\label{section:affine}
 The equivalence between Adams stacks and geometric morphisms suggests that it should be possible to characterize various geometric properties and constructions purely categorically. For example, in \cite{SCHAEPPI_TENSOR} the author showed that products of Adams stacks are sent to coproducts of tensor categories. In this section we set up the machinery necessary to give a
 categorical characterization of affine morphisms between Adams stacks.

 The result we seek to generalize is Serre's criterion for a morphism to be affine (see Corollary~\ref{cor:serres_criterion}), though to prove it we will need to extend the result of \cite{SCHAEPPI_TENSOR} to arbitrary fiber products (see Theorem~\ref{thm:pushouts}). In fact, these two results are closely linked: in order to prove Theorem~\ref{thm:pushouts}, we need to use basic results about affine morphisms. The following Definition extends \cite[Definition~3.1]{ALPER} to general tensor categories.

\begin{dfn}\label{def:cohomologically_affine}
 A tensor functor $F \colon \ca{A} \rightarrow \ca{B}$ between two tensor categories is called \emph{cohomologically affine} if the right adjoint of $F$ is faithful and right exact.
\end{dfn}

 Recall that an \emph{algebra} $A$ in a tensor category $\ca{C}$ consists of an object $A$, together with a \emph{multiplication} $\mu \colon A\otimes A \rightarrow A$ and a \emph{unit} $\eta \colon I \rightarrow A$, subject to the associativity and unit axioms. An algebra is called \emph{commutative} if the diagram
\[
 \xymatrix{A\otimes A \ar[rd]_{\mu} \ar[rr]^{s_{A,A}} && A \otimes A \ar[ld]^{\mu} \\ & A }
\]
 is commutative.

 A \emph{module} of an algebra $A \in \ca{C}$ is an object $M$, together with a morphism $A\otimes M \rightarrow M$ which is compatible with the unit and multiplication of $A$. A morphism of modules is a morphism $M \rightarrow N$ in $\ca{C}$ such that the diagram
\[
 \xymatrix{A\otimes M \ar[r] \ar[d] & A \otimes N \ar[d] \\ M \ar[r] & N} 
\]
 is commutative. 

\begin{notation}\label{notation:module_category}
 Let $A \in \ca{C}$ be a commutative algebra. The category of $A$-modules in $\ca{C}$ is denoted by $\ca{C}_A$. This is again a tensor category, with tensor product given by the evident coequalizer of the two actions of $A$.
\end{notation}

\begin{rmk}
 Let $\ca{C}$ be a pre-geometric category and let $\ca{A}$ be a commutative algebra. Then $\ca{C}_A$ is pre-geometric. Indeed, every module is a quotient of a free module, and every free module is a colimit of free modules on objects with duals. The category $\ca{C}_A$ is abelian since limits and colimits in $\ca{C}_A$ are computed as in $\ca{C}$.
\end{rmk}
 
\begin{example}\label{example:free_module_affine}
 Let $\ca{C}$ be a tensor category, and let $A$ be a commutative algebra in $\ca{C}$. Then the tensor functor $A \otimes - \colon \ca{C} \rightarrow \ca{C}_A$ which sends an object $M$ to the free $A$-module is cohomologically affine.
\end{example}

 In fact, we will see that every cohomologically affine functor between geometric categories is of this form. In order to state this, it is convenient to introduce a name for the functors in Example~\ref{example:free_module_affine}.

\begin{dfn}\label{def:affine_functor}
 A tensor functor $F \colon \ca{C} \rightarrow \ca{D}$ is called an \emph{affine functor} if there exists a commutative algebra $A$ in $\ca{C}$ and an equivalence $\ca{D} \simeq \ca{C}_A$ of tensor categories such that the diagram
\[
 \xymatrix{ & \ca{C} \ar[rd]^{A \otimes -} \ar[ld]_{F} \\ \ca{D} \ar[rr]_{\simeq}  && \ca{C}_A }
\]
 commutes up to isomorphism.
\end{dfn}

\begin{rmk}\label{rmk:algebra_from_affine_functor}
 Note that if there exists a commutative algebra $A$ as in Definition~\ref{def:affine_functor}, then $A$ is necessarily isomorphic to $U(I)$, where $U$ denotes the right adjoint of the tensor functor $F$ and $I \in \ca{D}$ is the unit object.
\end{rmk}
 
 Recall from \cite{FAUSK_HU_MAY} that a tensor functor $F \colon \ca{C} \rightarrow \ca{D}$ is said to \emph{satisfy the projection formula} or to be \emph{coclosed} if the dashed arrow making the diagram
\begin{equation}\label{eqn:projection_formula}
\vcenter{
\xymatrix@!C=100pt{  UF(X \otimes UY) \ar[r]^{U\varphi^{-1}_{X,UY}} &U(FX \otimes FUY) \ar[d]^{U(FX \otimes \varepsilon_Y)} \\ X \otimes UY \ar[u]^{\eta_{X \otimes UY}} \ar@{-->}[r]_-{\chi_{X,Y}} & U(FX \otimes Y) }
}
\end{equation}
 commutative is an isomorphism for every $X \in \ca{C}$ and $Y \in \ca{D}$. 
 
\begin{lemma}\label{lemma:duals_implies_projection}
 Suppose that $X \in \ca{C}$ has a dual $X^{\vee}$. Then the natural morphism $\chi_{X,Y}$ defined in Equation~\ref{eqn:projection_formula} is an isomorphism for all $Y \in \ca{C}$.
\end{lemma} 

\begin{proof}
 This is proved in \cite[Proposition~1.12]{FAUSK_HU_MAY}. %in the case where all objects have duals, but the argument works in general. Note that we have the sequence
%\begin{align*}
%\ca{C}(C,X \otimes UY) &\cong \ca{C}(C\otimes X^{\vee}, UY) \\
%&\cong \ca{C}(FC \otimes FX^{\vee},Y) \\
%&\cong \ca{C}(FC,FX \otimes Y) \\
%&\cong \ca{C}\bigr(C,U(FX\otimes Y)\bigl)
%\end{align*}
%of natural isomorphisms. Thus it suffices to show that this composite is induced by $\chi_{X,Y}$. In other words, we need to show that the image of the identity on $X \otimes UY$ under the above isomorphism is equal to $\chi_{X,Y}$. 
%
% This follows from the fact that we have used both the evaluation and coevaluation of the dual pair $X \dashv X^{\vee}$, and (after applying naturality conditions) we find that these two morphisms compose to the identity by one of the triangular identities.
\end{proof}

\begin{prop}\label{prop:cohomologically_affine_implies_affine}
 Let $F \colon \ca{C} \rightarrow \ca{D}$ be a cohomologically affine functor. If $\ca{C}$ is pre-geometric, then $F$ is coclosed and affine.
\end{prop}

\begin{proof}
 First note that exactness of $U$ implies that $U$ preserves all small colimits. Indeed, the left adjoint $F$ preserves finitely presentable objects by Lemma~\ref{lemma:pregeometric_implies_finitary}, and this implies that $U$ preserves all filtered colimits. A functor which preserves filtered colimits and finite colimits preserves all small colimits.

 Using this, we can show that $F$ is coclosed. Indeed, both the domain and codomain of the natural transformation
 \[
 \chi_{X,Y} \colon X\otimes UY \rightarrow U(FX \otimes Y)
 \]
 are cocontinuous in $X$. The claim now follows from the fact that every object in $\ca{C}$ is a colimit of objects with duals and from Lemma~\ref{lemma:duals_implies_projection}.
 
 To see that $F$ is affine, note that $U$ is exact and faithful, so it is monadic by Beck's Monadicity Theorem (recall from Definition~\ref{def:pregeometric_category} that pre-geometric categories are assumed to be abelian). It only remains to check that the symmetric monoidal monad in question is induced by tensoring with a commutative algebra in $\ca{C}$. The existence of the isomorphism
 \[
 \xymatrix{UFY \ar[r]^-{\cong} & U(FY\otimes I) \ar[r]^-{\chi^{-1}} & Y \otimes U(I)}
 \]
 shows that the underlying symmetric monoidal functor of the symmetric monoidal monad $UF$ is given by tensoring with the commutative algebra $U(I)$. The monad structure endows $U(I)$ with a compatible algebra structure, so by the Eckmann-Hilton argument, the monad structure must also be given by the algebra structure of $U(I)$.
\end{proof}

 In \cite{TOTARO}, Totaro showed that for certain class of stacks $X$, having the strong resolution property implies that the diagonal morphism is affine. In other words, a morphism with affine domain and target $X$ is affine. The following proposition shows that an analogous fact is true in the dual world of tensor categories: if the objects with duals form a generator, then any tensor functor with affine codomain is itself affine.
 
\begin{prop}\label{prop:affine_target_implies_affine}
 Let $\ca{C}$ be a pre-geometric category, $B$ a commutative $R$-algebra. Then any tensor functor $F \colon \ca{C} \rightarrow \Mod_B$ is affine.
\end{prop}

\begin{proof}
 By Proposition~\ref{prop:cohomologically_affine_implies_affine}, it suffices to show that $F$ is cohomologically affine. Thus we have to show that the right adjoint $U$ of $F$ is faithful and right exact.

 To see that it is exact, it suffices to show that it preserves epimorphisms (any right adjoint is left exact). Thus let $p \colon M \rightarrow N$ be an epimorphism of $B$-modules, let $V \in \ca{C}$ be an object with dual, and let $V \rightarrow UN$ be an arbitrary morphism. By adjunction, this last morphism corresponds to a morphism $FV \rightarrow N$. Since tensor functors preserve objects with duals, $FV$ is a finitely generated projective module. The morphism $FV \rightarrow N$ therefore factors through $p$. Applying the adjunction isomorphisms again we find that $V \rightarrow UN$ factors through $Up \colon UM \rightarrow UN$.
 
 The claim that $Up$ is an epimorphism now follows from the fact that the morphisms $V \rightarrow UN$ are jointly epimorphic (since the objects with duals form a generator of $\ca{C}$).
 
 It remains to show that $U$ is faithful. Since $U$ is exact this boils down to showing that $UM=0$ implies $M=0$ for every $B$-module $M$. This follows from the sequence
 \[
 M \cong \Mod_B(B,M) \cong \Mod_B(FI,M) \cong \ca{C}(I,UM)
 \]
 of isomorphisms.
% Consider the inclusion $K \colon \ca{C}^d \rightarrow \ca{C}$ of objects with duals and the induced triangle
% \[
 %\xymatrix@!=30pt{\ca{C} \ar@<-0.5ex>[rd]_{F} \ar@<-0.5ex>[rr]_{\widetilde{K}} & & \mathcal{P}\ca{C}^d \ar@<-0.5ex>[ll]_{\Lan_Y K}  \ar@<0.5ex>[ld]^{\Lan_Y FK} \\ & \Mod_B \ar@<0.5ex>[ru]^{\widetilde{FK}} \ar@<-0.5ex>[lu]_{U}}
% \]
 %of adjunctions. The triangle of left adjoints commutes up to natural isomorphism
 %\[
 %F \circ \Lan_Y K \cong \Lan_Y FK
 %\]
 %(since left adjoints preserve left Kan extensions), so the triangle of right adjoints must also commute up to natural isomorphism $\widetilde{K} U \cong \widetilde{FK}$. The assumption that $\ca{C}$ is pre-geometric implies that $K$ is dense (see Remark~\ref{rmk:duals_dense}), which is equivalent to the fact that $\widetilde{K}$ is fully faithful.
%
%  This last observation reduces the question to showing that $\widetilde{FK}$ is faithful and right exact. By definition, this functor sends $M \in \Mod_B$ to the presheaf
%\[
% \Mod_B(FK-,M) \smash{\rlap{,}}
%\]
% so it suffices to check that each functor $\Mod_B(FD,-)$ where $D \in \ca{C}$ is an object with a dual is cocontinuous, and that at least one of them is faithful. Since $F$ is a tensor functor, it preserves duals. Thus $FD$ is finitely generated and projective, and it follows that $\Mod_B(FD,-)$ is indeed cocontinuous. Moreover, the isomorphism $\varphi^F_0 \colon B \rightarrow FI$ induces a natural isomorphism
% \[
% M \cong \Mod_B(B,M) \cong \Mod_B(FI,M) \smash{\rlap{,}}
% \]
 %from which we deduce that $\Mod_B(FI,-)$ is faithful.
\end{proof}

 Proposition~\ref{prop:cohomologically_affine_implies_affine} allows us to identify (cohomologically) affine morphisms with pre-geometric domain $\ca{C}$ with commutative algebras in $\ca{C}$. In other words, the entire information about an affine morphism (including its codomain) is contained in a single object in $\ca{C}$. 
 
 Since all functors with affine codomain are affine by Proposition~\ref{prop:affine_target_implies_affine}, we get the following intrinsic characterization of such functors. As a corollary we obtain an alternative description of faithfully flat coverings of pre-geometric categories (see Corollary~\ref{cor:faithfully_flat_coverings_vs_algebras}). This description is the starting point for the proofs of the two main results of this paper (Theorem~\ref{thm:fiber_product} and Theorem~\ref{thm:description}).
 
\begin{dfn}\label{def:affine_algebra}
 Let $\ca{C}$ be a pre-geometric category. A commutative algebra $A \in \ca{C}$ is called an \emph{affine algebra} if $A$ is a projective generator of the category $\ca{C}_A$ of $A$-modules.
\end{dfn}

\begin{example}
 If $\ca{C}$ is the category $\QCoh(X)$ of quasi-coherent sheaves on a quasi-compact and quasi-separated scheme $X$, then a commutative algebra in $\ca{C}$ is precisely a quasi-coherent $\mathcal{O}_X$-algebra $\mathcal{A}$. Such an algebra is affine if and only if the relative spectrum $\Spec_X(\mathcal{A})$ is an affine scheme.
\end{example}

\begin{prop}\label{prop:affine_functor_affine_algebra}
 Let $\ca{C}$ be a pre-geometric category. There is an equivalence between tensor functors $F \colon \ca{C} \rightarrow \Mod_B$ where $B$ is a commutative $R$-algebra on the one hand and affine algebras in $\ca{C}$ on the other. Under this equivalence, a tensor functor $F$ corresponds to the affine algebra $U(B)$, where $U$ denotes the right adjoint of $F$. 
\end{prop}

\begin{proof}
 From Proposition~\ref{prop:affine_target_implies_affine} we know that such tensor functors are (up to equivalence) of the form $A \otimes - \colon \ca{C} \rightarrow \ca{C}_A$. It is of the above form if and only if the tensor category $\ca{C}_A$ is equivalent to a category of $B$-modules for some commutative $R$-algebra $B$. The claim therefore follows from Lemma~\ref{lemma:projective_generator_implies_affine} below.
 
 The second statement follows directly from Remark~\ref{rmk:algebra_from_affine_functor}.
\end{proof}

\begin{lemma}\label{lemma:projective_generator_implies_affine}
 Let $\ca{C}$ be a tensor category such that the unit $I$ is a projective generator, and let $B=\ca{C}(I,I)$ be the commutative $R$-algebra of endomorphisms of $I$. Then the functor $\ca{C}(I,-) \colon \ca{C} \rightarrow \Mod_B$ is an equivalence of symmetric monoidal categories.
\end{lemma}

\begin{proof}
 Since $I$ is by assumption finitely presentable (see Definition~\ref{def:tensor_category}), the right adjoint functor $\ca{C}(I,-)$ is exact and preserves filtered colimits, so it is in fact cocontinuous. Using this, the result is well-known at the level of underlying categories. Compatibility with the symmetric monoidal structure follows since tensor products in both categories commute with colimits.
\end{proof}

 As an immediate consequence, we get the following intrinsic characterization of faithfully flat coverings of pre-geometric categories. Recall that a commutative algebra $A$ in a tensor category $\ca{C}$ is called \emph{flat} if the functor $A \otimes -$ is exact, and \emph{faithful} if $A\otimes - \colon \ca{C} \rightarrow \ca{C}$ is faithful.

\begin{cor}\label{cor:faithfully_flat_coverings_vs_algebras}
 Let $\ca{C}$ be a pre-geometric category. Then the assignment which sends a faithfully flat covering
\[
 F \colon \ca{C} \rightarrow \Mod_B
\]
 to the algebra $U(B)$ (where $U$ denotes the right adjoint of $F$) gives an equivalence between faithfully flat coverings of $\ca{C}$ and faithfully flat affine algebras in $\ca{C}$.
\end{cor}

\begin{proof}
 By Proposition~\ref{prop:affine_functor_affine_algebra}, $F$ is (up to equivalence) given by tensoring with the affine algebra $U(B)$. Thus it is faithful and exact if and only if this algebra is faithfully flat.
\end{proof}

 The following two lemmas are useful for recognizing affine algebras respectively faithfully flat algebras in pre-geometric categories.
 
\begin{lemma}\label{lemma:affine_algebra_characterization}
 Let $\ca{C}$ be a pre-geometric category. Then a commutative algebra $A \in \ca{C}$ is affine if and only if the functor
 \[
 \ca{C}(I,U-) \colon \ca{C}_A \rightarrow \Mod_R
\] 
is faithful and right exact, where $U \colon \ca{C}_A \rightarrow \ca{C}$ denotes the forgetful functor.
\end{lemma}

\begin{proof}
 This follows from the natural isomorphism $\ca{C}_A(A,-) \cong \ca{C}(I,U-)$.
\end{proof}

\begin{lemma}\label{lemma:faithfully_flat_condition}
 Let $\ca{C}$ be a pre-geometric category, and let $A \in \ca{C}$ be a commutative algebra. Then $A$ is faithfully flat if and only if the unit morphism $\eta \colon I \rightarrow A$ is a monomorphism and the cokernel of $\eta$ is flat.
\end{lemma}

\begin{proof}
 Recall from Remark~\ref{rmk:flat_resolutions_hence_tame} that $\ca{C}$ is a tame abelian category in the sense of \cite[Remark~5.3]{LURIE}. The statement is proved for tame tensor categories in \cite[Lemma~5.5]{LURIE}. It also follows directly from the proof of Lemma~\ref{lemma:flatness_from_filtration} below.
\end{proof}

\section{Quasi-coherent sheaves on fiber products}\label{section:fiber}

\subsection{Proof of Theorem~\ref{thm:fiber_product}}\label{section:main_proof}
 In this section we will show that the pseudofunctor which sends an Adams stack to its category of finitely presentable quasi-coherent sheaves sends fiber products to bicategorical pushouts. Recall that $\ca{AS}$ denotes the 2-category of Adams stacks and $\ca{RM}$ denotes  the 2-category of right exact symmetric monoidal categories.

\begin{dfn}\label{def:tensor_product_rex}
 Let $\ca{A}$, $\ca{B}$, and $\ca{C}$ be right exact monoidal categories, and let $f \colon \ca{C} \rightarrow \ca{A}$ and $g \colon \ca{C} \rightarrow \ca{B}$ be right exact symmetric strong monoidal functors. Then we write $\ca{A} \boxten{\ca{C}} \ca{B}$ for the codescent object (or geometric realization, see for example \cite{LACK_CODESCENT} for a definition) of the truncated pseudosimplicial diagram
 \[
 \xymatrix{\ca{A} \boxtimes \ca{C} \boxtimes \ca{C} \boxtimes \ca{B} \ar@<1ex>[r] \ar[r] \ar@<-1ex>[r] & \ca{A} \boxtimes \ca{C} \boxtimes \ca{B} \ar@<0.75ex>[r] \ar@<-0.75ex>[r] & \ca{A} \boxtimes \ca{B} \ar[l] }
 \]
 in $\ca{RM}$, with faces induced by $f$, $g$, and the identity on $\ca{C}$, and degeneracy given by the unit object in $\ca{C}$.
\end{dfn}

\begin{thm}\label{thm:pushouts}
 The pseudofunctors
 \[
  \QCoh \colon \ca{AS}^{\op} \rightarrow \ca{T} \quad\text{and} \quad
  \QCoh_{\fp} \colon \ca{AS}^{\op} \rightarrow \ca{RM}
 \]
  send fiber products to bicategorical pushouts. In particular, there is an isomorphism
\begin{equation}\label{eqn:fiber_product_as_tensor}
   \QCoh_{\fp}(X \pb{Z} Y) \simeq \QCoh_{\fp}(X) \boxten{\QCoh_{\fp}(Z)} \QCoh_{\fp}(Y)
\end{equation}
  of symmetric monoidal $R$-linear categories.
\end{thm}

 Note that this theorem immediately implies Theorem~\ref{thm:fiber_product}.

\begin{proof}[Proof of Theorem~\ref{thm:fiber_product}]
 The claim of Theorem~\ref{thm:fiber_product} is exactly the existence of the equivalence~\eqref{eqn:fiber_product_as_tensor}.
\end{proof}

 To prove the equivalence in Formula~\eqref{eqn:fiber_product_as_tensor} from the first assertion, it suffices to prove that the tensor product defined in Definition~\ref{def:tensor_product_rex} is equivalent to the bicategorical pushout in $\ca{RM}$. Using the fact that $\boxtimes$ is a bicategorical coproduct (see \cite{SCHAEPPI_TENSOR}), this follows from the following general fact about 2-categories. 

\begin{prop}\label{prop:pushout_in_general_2_category}
 Let $\ca{K}$ be a 2-category with finite bicategorical coproducts (denoted by $+$) and bicategorical codescent objects of truncated pseudosimplicial diagrams. Let $f \colon C \rightarrow A$ and $g \colon C \rightarrow B$ be 1-cells in $\ca{K}$. Then the codescent object of the pseudosimplicial diagram
\begin{equation}\label{eqn:codescent_diagram}
 \xymatrix{A + C + C + B \ar@<1ex>[r] \ar[r] \ar@<-1ex>[r] & A + C + B \ar@<0.75ex>[r] \ar@<-0.75ex>[r] & A + B \ar[l] }
\end{equation}
 is a bicategorical pushout of $f$ along $g$. The faces in Diagram~\eqref{eqn:codescent_diagram} are built from $f$, $g$ and identities (using at most one of $f$ and $g$), and the degeneracy is given by the coproduct inclusion.
\end{prop}

\begin{proof}
 To give a cocone on this diagram amounts to giving a 1-cell $x \colon A + B \rightarrow X$ and an invertible 2-cell
 \[
 \xymatrix{A+C+B \xtwocell[1,1]{}\omit{\alpha} \ar[r] \ar[d] & A+B \ar[d]^x \\ A+B \ar[r]_-{x} & X}
 \]
 subject to two conditions. Using the universal property of coproducts one finds that this is equivalent to giving two 1-cells $f^{\prime} \colon B \rightarrow X$ and $g^{\prime} \colon A \rightarrow X$, together with an isomorphism
 \[
 \xymatrix{C \xtwocell[1,1]{}\omit{\beta} \ar[r]^{f} \ar[d]_{g} & A \ar[d]^{g^{\prime}} \\ B \ar[r]_-{f^{\prime}} & X}
 \]
 not subject to any conditions. Thus the codescent object of Diagram~\eqref{eqn:codescent_diagram} is indeed a bicategorical pushout.
\end{proof}

 Before proving Theorem~\ref{thm:pushouts}, we point out the following important corollary.

\begin{cor}\label{cor:bilimits}
 The pseudofunctors $\QCoh(-)$ and $\QCoh_{\fp}(-)$ preserve finite weighted bilimits.
\end{cor}

\begin{proof}
 Note that the $\QCoh\bigr(\Spec(R)\bigl) \simeq \Mod_R$ is the initial $R$-linear tensor category, so both pseudofunctors send the terminal object to the initial object. Thus it follows from Theorem~\ref{thm:pushouts} that $\QCoh(-)$ and $\QCoh_{\fp}(-)$ preserve finite conical limits (that is, they send them to finite conical colimits). Indeed, the bicategorical pullback diagram
 \[
 \xymatrix{ E \ar[r] \ar[d] & Y \ar[d]^{\Delta} \\ X \ar[r]_-{(f,g)} & Y\times Y }
 \]
 gives the bicategorical equalizer of $f,g \colon X \rightarrow Y$.
 
 To get preservation all weighted colimits it therefore suffices to show that the pseudofunctors also preserves powers (also known as cotensors) of finitely presentable categories (see \cite{STREET_CORRECTIONS}). Since natural transformations between tensor functors whose domain is a geometric category are always invertible (see \cite[Proposition~5.2.3]{SAAVEDRA}), it suffices to prove this for finitely presentable groupoids. This follows from the fact that ``duals invert:'' the component of a monoidal natural transformation at an object with dual is invertible, and the claim follows since any object can be built out of objects with duals using (finite) colimits. 

 Thus it suffices to show that the power $X^G$ is preserved where $X \in \ca{AS}$ and $G$ is a finitely presentable groupoid. Since every such groupoid can be built as a finite colimit from the groupoid $\Aut$ (consisting of a single object with automorphism group $\mathbb{Z}$), it suffices to consider the case $G=\Aut$. The power $X^{\Aut}$ classifies automorphisms of 1-cells with codomain $X$ and is also known as the \emph{inertia stack} of $X$. The inertia stack can be computed as pullback of the diagonal along itself. Thus preservation of pullbacks implies preservation of powers by finitely presentable groupoids. 
\end{proof}

 The proof of Theorem~\ref{thm:pushouts} follows the same strategy used in \cite{SCHAEPPI_TENSOR} to show that the pseudofunctor $\QCoh_{\fp}(-)$ sends products to coproducts. However, since general pushouts are not as tractable as coproducts, we prove this result in two steps. First we study pushouts along affine functors. The case of general fiber products reduces to this: the fiber product of $f$ and $g$ can be computed as the pullback
 \[
 \xymatrix{X \pb{Z} Y \ar[r] \ar[d] & Z \ar[d]^{\Delta} \\ X \times Y \ar[r]_{f\times g} & Z \times Z}
 \]
 of $f \times g$ along the diagonal, and Adams stacks have affine diagonal. More precisely, we use the following (sometimes technical) results.

 Recall from Definition~\ref{def:affine_functor} that an affine functor is a tensor functor which is equivalent to the free $A$-module functor $A \otimes - \colon \ca{C} \rightarrow \ca{C}_A$ for some commutative algebra $A$ in $\ca{C}$. The following proposition therefore allows us to compute pushouts of tensor categories along affine functors. We prove it in \S \ref{section:pushout_along_affine}.
 
\begin{prop}\label{prop:pushout_along_affine}
 Let $F \colon \ca{C} \rightarrow \ca{D}$ be a tensor functor, and let $A \in \ca{C}$ be a commutative algebra. Then the diagram
 \[
 \xymatrix{\ca{C}_A  \ar[r]^-{\overline{F}} & \ca{D}_{FA} \\ \ca{C} \ar[u]^{A\otimes -} \ar[r]_-F & \ca{D} \ar[u]_{FA\otimes-}}
 \] 
 is a bicategorical pushout, where $\overline{F}$ denotes the functor which sends an $A$-module $M$ to the $FA$-module $FM$.
\end{prop}

 If $f \colon X \rightarrow Z$ and $g \colon Y \rightarrow Z$ are morphisms in $\ca{AS}$, then
 \[
 \xymatrix{ \QCoh(X) \ar[r] & \QCoh(X \pb{Z} Y) \\ \QCoh(Z) \ar[u]^{f^{\ast}} \ar[r]_{g^{\ast}} & \QCoh(Y) \ar[u] }
 \]
 is a pushout square \emph{in the image of} $\QCoh(-)$ by Theorem~\ref{thm:geometric_equals_qc}. If $f^{\ast}$ is an affine functor, we can compute the pushout among \emph{all} tensor categories using Proposition~\ref{prop:pushout_along_affine}. Thus, to show that the diagram above is a pushout among all tensor categories it suffices to prove the following proposition, which we do in \S \ref{section:closure_under_affine_functors}.
 
 \begin{prop}\label{prop:closure_under_affine_functors}
 Let $\ca{C}$ be a geometric category, $\ca{D}$ an arbitrary tensor category, and let $F \colon \ca{C} \rightarrow \ca{D}$ be a cohomologically affine functor. Then $\ca{D}$ is geometric.
\end{prop}

\begin{rmk}
 Let $\ca{A}$ be a Tannakian category and let $A \in \Ind(\ca{A})$ be a commutative algebra. Proposition~\ref{prop:closure_under_affine_functors} implies that the category of $A$-modules is the category of quasi-coherent sheaves on an Adams stack. Moreover, note that tensor functors between module categories which are compatible with the free functors are induced by unique morphisms of algebras (see Lemma~\ref{lemma:module_category_universal_property}).

 This allows us to identify the category of affine $\ca{A}$-schemes defined in \cite[\S 7.8]{DELIGNE} (that is, the opposite of the category of commutative algebras in $\Ind(\ca{A})$) with a subcategory of Adams stacks over the gerbe associated to $\ca{A}$.
\end{rmk}

 As mentioned above, Propositions~\ref{prop:pushout_along_affine} and \ref{prop:closure_under_affine_functors} have the following consequence.
 
 \begin{cor}\label{cor:pushout_along_affine_functor}
 The pseudofunctor $\QCoh \colon \ca{AS}^{\op} \rightarrow \ca{T}$ preserves pullbacks along morphisms $f \colon X \rightarrow Z$ with the property that $f^{\ast}$ is an affine functor.
\end{cor}

\begin{proof}
 Both $\QCoh(X \pb{Z} Y)$ and the bicategorical pushout computed in Proposition~\ref{prop:pushout_along_affine} are bicategorical pushouts in the image of $\QCoh(-)$ (the former since $\QCoh(-)$ is an embedding, the latter by Proposition~\ref{prop:closure_under_affine_functors}). Since they have the same universal property, they must be equivalent.
\end{proof}

 Now, to conclude the proof of Theorem~\ref{thm:pushouts}, we would like to apply this reasoning to the diagonal morphism of an Adams stack. To do this, we need to know that a morphism $f$ of Adams stacks is affine if and only if $f^{\ast}$ is an affine functor. By Proposition~\ref{prop:cohomologically_affine_implies_affine}, it suffices to check that $f^{\ast}$ is cohomologically affine. For Artin stacks, this is proved in \cite[Proposition~3.9]{ALPER}, using some standard results like flat base change. We give a direct proof using the above corollary in \S \ref{section:affine_morphism_implies_affine_functor}.

\begin{prop}\label{prop:affine_morphism_implies_affine_functor}
 Let $f \colon X \rightarrow Y$ be an affine morphism between Adams stacks. Then the inverse image functor
\[
f^{\ast } \colon  \QCoh(Y) \rightarrow \QCoh(X)
\] 
 is an affine functor. 
\end{prop}

 With all the ingredients in place, we are now ready to prove Theorem~\ref{thm:pushouts}.

\begin{proof}[Proof of Theorem~\ref{thm:pushouts}.]
 The (bicategorical) pullback of $f \colon X \rightarrow Z$ along $g\colon Y \rightarrow Z$ in the 2-category of Adams stacks can be computed as the pullback
 \[
 \xymatrix{X \pb{Z} Y \ar[r] \ar[d] & Z \ar[d]^{\Delta} \\ X \times Y \ar[r]_{f\times g} & Z \times Z}
 \]
 of $f\times g$ along the diagonal of $Z$ in the 2-category of \emph{all} $\fpqc$-stacks. Indeed, since $Z$ is an Adams stack, it is the stack associated to a flat affine groupoid. Therefore it has affine diagonal $\Delta$, and it follows that $X \pb{Z} Y$ is again an Adams stack (see \cite[Corollary~4.7]{SCHAEPPI_TENSOR}).
 
 Moreover, $\Delta^{\ast}$ is an affine functor by Proposition~\ref{prop:affine_morphism_implies_affine_functor}. By Corollary~\ref{cor:pushout_along_affine_functor}, $\QCoh(-)$ preserves the above pullback diagram. It remains to check that the resulting pushout in $\ca{T}$ is the pushout of $f^{\ast}$ along $g^{\ast}$. To do this, it suffices to check that $\QCoh(Z\times Z)$ is equivalent to the bicategorical coproduct of $\QCoh(Z)$ with itself, and that $\Delta^{\ast}$ is the codiagonal morphism.
 
 We know from \cite[Theorems~1.7 and 5.2]{SCHAEPPI_TENSOR} that $\QCoh_{\fp}(-)$ sends binary products to binary coproducts in the 2-category $\ca{RM}$ of right exact symmetric monoidal categories. Note that tensor functors out of a pre-geometric category have to preserve finitely presentable objects (see Lemma~\ref{lemma:pregeometric_implies_finitary}). Using this, it follows that $\QCoh(Z\times Z)$ is a bicategorical coproduct in the 2-category $\ca{T}$ of tensor categories. This concludes the proof that $\QCoh(-)$ preserves fiber products.
 
 To prove that $\QCoh_{\fp} \colon \ca{AS}^{\op} \rightarrow \ca{RM}$ has the same property, we can again appeal to Lemma~\ref{lemma:pregeometric_implies_finitary}: the equivalence
 \[
 \ca{RM}(\QCoh_{\fp}(X),\ca{A}) \simeq  \ca{T}\bigr(\QCoh(X),\Ind(\ca{A})\bigl)
 \]
 shows that $\QCoh_{\fp}(-)$ preserves all the limits that $\QCoh(-)$ preserves. 

 As already mentioned at the beginning of this section, the equivalence
 \[
 \QCoh_{\fp}(X \pb{Z} Y) \simeq \QCoh_{\fp}(X) \boxten{\QCoh_{\fp}(Z)} \QCoh_{\fp}(Y)
 \]
 of symmetric monoidal categories now follows from Proposition~\ref{prop:pushout_in_general_2_category}.
\end{proof}

 As a corollary of the proof of Theorem~\ref{thm:pushouts} we obtain a proof of Serre's criterion for affine morphisms between Adams stacks.

\begin{cor}\label{cor:serres_criterion}
 A morphism of Adams stacks is affine if and only if $f^{\ast}$ is a cohomologically affine functor, that is, if $f_{\ast}$ is exact and faithful.
\end{cor} 

\begin{proof}
 One direction follows from Proposition~\ref{prop:affine_morphism_implies_affine_functor}. To see the converse direction, let $f \colon X \rightarrow Y$ be a morphism of Adams stacks such that $f^{\ast}$ is cohomologically affine. From Proposition~\ref{prop:cohomologically_affine_implies_affine} we know that $f^{\ast}$ is an affine functor, that is, it is equivalent to $A\otimes -$ for some commutative algebra $A \in \QCoh(Y)$.

 Let $g \colon Z \rightarrow Y$ be a morphism of Adams stack with affine domain $Z=\Spec(B)$. From Proposition~\ref{prop:pushout_along_affine}, applied to the functor $F=g^{\ast}$, it follows that the bicategorical pushout of $f^{\ast}$ along $g^{\ast}$ is given by $(\Mod_B)_{g^{\ast}A}$, that is, by the category of modules of a commutative $B$-algebra. Since $\QCoh(X \pb{Y} Z)$ is equivalent to this category (see Theorem~\ref{thm:pushouts}), this shows that $X \pb{Y} Z$ is affine. 
\end{proof}

\subsection{Bicategorical pushouts along affine functors}\label{section:pushout_along_affine}
 To prove Proposition~\ref{prop:pushout_along_affine} we have to check that a certain diagram in a 2-category is a bicategorical pushout diagram. In particular, we have to check that the isomorphisms we write down are natural (in the 2-categorical sense). This can either be done ``by hand,'' by checking that all the necessary diagrams commute. We have instead opted to use some more formal techniques that minimize this kind of computation. One of these uses the notion of isofibrations in the category of small categories. A functor $F \colon \ca{C} \rightarrow \ca{C}^{\prime}$ is called an \emph{isofibration} if any isomorphism $A^{\prime} \cong FA$ already lies in the image of $F$. For example, the forgetful functor from groups to sets is an isofibration since the group structure can be transferred along any bijection of sets.
 
 The standard fact about isofibrations that we will use is that the strict pullback along an isofibration is equivalent to the bicategorical pullback. This implies in particular that, given a diagram of functors and equivalences
 \[
 \xymatrix{ \ca{A} \ar[r]^-{F} \ar[d]^{\simeq} & \ca{B} \ar[d]^{\simeq} & \ca{C} \ar[d]^{\simeq} \ar[l]_{G} \\ 
  \ca{A}^{\prime} \ar[r]_{F^{\prime}}  & \ca{B}^{\prime} & \ca{C}^{\prime} \ar[l]^{G}}
 \]
 where both $F$ and $F^{\prime}$ are isofibrations, the induced functor between strict pullbacks $\ca{A} \pb{\ca{B}} \ca{C} \rightarrow \ca{A}^{\prime} \pb{\ca{B}^{\prime}} \ca{C}^{\prime}$ is again an equivalence.

 The key ingredient in proving Proposition~\ref{prop:pushout_along_affine} is the following lemma, which appears for example as \cite[Proposition~4.2]{BRANDENBURG_TENSOR}. Let $\ca{C}$ be a tensor category. The lemma states that the following two 2-functors
\[
 P,Q \colon \ca{T} \rightarrow \Cat
\]
 are equivalent.
 
\begin{dfn}\label{def:pseudofunctors_for_pushout}
 Fix a tensor category $\ca{C}$, and let $A \in \ca{C}$ be a commutative algebra in $\ca{C}$. The 2-functor $P_{\ca{C},A}$ sends a tensor category $\ca{E}$ to the category with objects the triples $(F,\overline{F},\sigma)$ where $\sigma$ is a natural isomorphism
 \[
 \xymatrix{& \ca{C}_A \ar[rd]^{\overline{F}} \\ \ca{C} \ar[ru]^{A \otimes -} \ar[rr]_-{F}  \xtwocell[0,2]{}\omit{^<-3>\sigma} && \ca{E}}
 \]
 of tensor functors, and morphisms $(F,\overline{F},\sigma) \rightarrow (G,\overline{G},\tau)$ the pairs of natural transformations $\alpha \colon F \Rightarrow G$ and $\overline{\alpha} \colon \overline{F} \Rightarrow \overline{G}$ which are compatible with $\sigma$ and $\tau$. The 2-functor structure on $P$ is given by composing with tensor functors $\ca{E} \rightarrow \ca{E}^{\prime}$.
 
 The 2-functor $Q_{\ca{C},A} \colon \ca{T} \rightarrow \Cat$ sends a tensor category $\ca{E}$ to the category with objects the pairs $(F, f)$ consisting of a tensor functor $F \colon \ca{C} \rightarrow \ca{E}$ and a morphism of commutative algebras $f \colon FA \rightarrow I$, and morphisms $(F,f) \rightarrow (G,g)$ the natural transformations $\alpha \colon F \Rightarrow G$ of tensor functors for which the diagram
 \[
 \xymatrix{FA \ar[rd]_{f} \ar[rr]^{\alpha_A} && GA \ar[ld]^{g}\\ & I}
 \]
 is commutative. A tensor functor $H \colon \ca{E} \rightarrow \ca{E}^{\prime}$ sends $(F,f)$ to $\bigr(HF,(\varphi_0^H)^{-1} Hf\bigl)$.
 \end{dfn}

 Note that for any commutative algebra $A$, the multiplication $\mu \colon A \otimes A \rightarrow A$ is a morphism of $A$-modules.

\begin{lemma}\label{lemma:module_category_universal_property}
 Let $\ca{C}$ be a tensor category, and let $A \in \ca{C}$ be a commutative algebra. The assignment which sends an object $(F,\overline{F},\sigma)$ of $P_{\ca{C},A}(\ca{E})$ to the pair $(F,f)$ where $f$ denotes the composite
 \[
 \xymatrix{FA \ar[r]^-{\sigma_A} & \overline{F}(A\otimes A) \ar[r]^-{\overline{F}(\mu)} & \overline{F}(A) \ar[r]^-{(\varphi_0^{\overline{F}})^{-1}} & I}
 \]
 gives a 2-natural equivalence between $P_{\ca{C},A}$ and $Q_{\ca{C},A}$.
\end{lemma} 
 
\begin{proof}
 It is straightforward to check that the assignment is 2-natural. To see that it is an equivalence, we give an explicit construction of the inverse. Starting with a tensor functor $F \colon \ca{C} \rightarrow \ca{E}$ and a morphism $f \colon FA \rightarrow I$ of commutative algebras, we construct a lift $\overline{F} \colon \ca{C}_A \rightarrow \ca{E}$ as follows. Recall that for every $A$-module $M$ with action $\rho \colon A\otimes M \rightarrow M$, the diagram
\[
 \xymatrix{ A\otimes A \otimes M \ar@<0.5ex>[r]^-{\mu \otimes M} \ar@<-0.5ex>[r]_-{A\otimes \rho} & A \otimes M \ar[r]^-{\rho} & M } 
\]
is a coequalizer diagram in $\ca{C}_A$. Therefore we have no choice but to define $\overline{F}M$ to be the coequalizer of the diagram
\[
\xymatrix{FA \otimes FM \ar[rr]^-{f \otimes FM} \ar[rd]_{\varphi^F_{A,M}} && FM \ar[r]^-{\xi_M} & \overline{F}M \\ & F(A\otimes M) \ar[ru]_{F\rho}}
\]
 in $\ca{E}$. The mate of $\xi$ under the adjunction $A\otimes - \dashv U \colon \ca{C} \rightleftarrows \ca{C}_A$, that is, the composite
 \[
  \xymatrix{FC \ar[r]^-{F\eta_C} & FU(A\otimes C) \ar[r]^-{\xi_{ A \otimes C }} & \overline{F}(A\otimes C)} \smash{\rlap{,}}
 \]
 defines a natural transformation $\sigma \colon F \Rightarrow \overline{F}(A \otimes -)$. As one of the referees pointed out, this construction can be described more conceptually as the tensor functor which sends an $A$-module $M$ to $I \ten{F(A)} F(M)$, where $F(M)$ is endowed with the natural $F(A)$-module structure and $I$ is endowed with an $F(A)$-algebra structure via $f$. 

 A lengthy but straightforward computation shows that $\sigma$ is an isomorphism, that $\overline{F}$ is a tensor functor, and that this construction is indeed inverse to the assignment described above.
\end{proof}

\begin{proof}[Proof of Proposition~\ref{prop:pushout_along_affine}.]
 The claim follows formally from applying Lemma~\ref{lemma:module_category_universal_property} twice. First note that the 2-natural transformation
\[
e_1 \colon \ca{T}(\ca{D}_{FA},-) \Rightarrow P_{\ca{D},FA}(-)
\]
 which sends $H \colon \ca{D}_{FA} \rightarrow \ca{E}$ to the strictly commuting diagram
\[
 \xymatrix{& \ca{D}_{FA} \ar[rd]^{H} \\ \ca{D} \ar[ru]^{FA \otimes -} \ar[rr]_{H(FA \otimes -)} \xtwocell[0,2]{}\omit{^<-3>\id} && \ca{E} } 
\]
 is an equivalence. From Lemma~\ref{lemma:module_category_universal_property}, we get an equivalence
\[
 e_2 \colon P_{\ca{D},FA}(-) \Rightarrow Q_{\ca{D},FA}(-) \smash{\rlap{.}}
\]
 To finish the proof, we use the fact that both $P_{\ca{C},A}(\ca{E})$ and $Q_{\ca{C},A}(\ca{E})$ have a natural forgetful functor to $\ca{T}(\ca{C},\ca{E})$ (which sends and object $(G,\overline{G},\sigma)$ of $P_{\ca{C},A}(\ca{E})$ to $G$, and an object $(G,\varphi)$ of $Q_{\ca{C},A}(\ca{E})$ to $G$). Note that the diagram
\[
 \xymatrix{Q_{\ca{D},FA}(\ca{E}) \ar[r] \ar[d] & Q_{\ca{C},A}(\ca{E}) \ar[d] \\ \ca{T}(\ca{D},\ca{E}) \ar[r]_{\ca{T}(F,\ca{E})} & \ca{T}(\ca{C},\ca{E})}
\]
 is a strict pullback diagram. On the other hand, the strict pullback of the diagram
\[
 \xymatrix{ & P_{\ca{C},A}(\ca{E}) \ar[d] \\ \ca{T}(\ca{D},\ca{E}) \ar[r]_{\ca{T}(F,\ca{E})} & \ca{T}(\ca{C},\ca{E})} 
\]
 represents the pushout of $F \colon \ca{C} \rightarrow \ca{D}$ along $A \otimes - \colon \ca{C} \rightarrow \ca{C}_A$, so to conclude the proof we only need to show that these two pullbacks are equivalent. 

 The equivalence of Lemma~\ref{lemma:module_category_universal_property} is compatible with the forgetful functors, so it induces a comparison morphism between the two strict pullbacks. Since both the forgetful functors are isofibrations, this comparison morphism must also be an equivalence.
\end{proof}

\subsection{Closure of geometric categories under affine functors}\label{section:closure_under_affine_functors}

 Fix a geometric category $\ca{C}$, and let $F \colon \ca{C} \rightarrow \ca{D}$ be a cohomologically affine functor, with right adjoint denoted by $U$. Note that this functor is in fact affine by Proposition~\ref{prop:cohomologically_affine_implies_affine}. Thus the image of $F$ --- the subcategory of free modules --- is a generator of $\ca{D}$. It follows that the image of a generating set of duals in $\ca{C}$ forms a generating set of duals in $\ca{D}$. This shows that $\ca{D}$ is pre-geometric.

 To show that $\ca{D}$ is geometric is therefore equivalent to showing that $\ca{D}$ contains a faithfully flat affine algebra (see Corollary~\ref{cor:faithfully_flat_coverings_vs_algebras}). Let $A \in \ca{C}$ be a faithfully flat affine algebra. We claim that $FA$ is a faithfully flat affine algebra in $\ca{D}$.

\begin{lemma}\label{lemma:faithfully_flat_stable}
 Let $\ca{C}$ be a geometric category, and let $A \in \ca{C}$ be a faithfully flat affine algebra. Let $F \colon \ca{C} \rightarrow \ca{D}$ be a tensor functor with pre-geometric codomain. Then $FA$ is faithfully flat.
\end{lemma}

\begin{proof}
 We check that the conditions of Lemma~\ref{lemma:faithfully_flat_condition} are satisfied, that is, that the sequence
\[
 \xymatrix{0 \ar[r] & I \ar[r]^-{\eta } & FA \ar[r] & FA\slash I \ar[r] & 0}
\]
 is exact and that $FA \slash I$ is flat.

 From the proof of \cite[Theorem~1.3.2]{SCHAEPPI_STACKS}, we know that the sequence
\[
 \xymatrix{0 \ar[r] & I \ar[r]^-{\eta } & A \ar[r] & A\slash I \ar[r] & 0} 
\]
 in $\ca{C}$ is a filtered colimit of sequences
\begin{equation}\label{eqn:exact_sequence_of_duals}
 \xymatrix{0 \ar[r] & I \ar[r]^-{\eta } & A_i \ar[r] & A_i\slash I \ar[r] & 0}   
\end{equation}
 of objects with duals. Since $\ca{C}$ is geometric, it follows that the dual sequence of~\eqref{eqn:exact_sequence_of_duals} is also exact. As shown\footnote{This was only shown in the special case $\ca{D}=\Mod_B$, but the proof works at the desired level of generality.} in the proof of \cite[Theorem~1.3.2]{SCHAEPPI_STACKS}, this implies that $F$ preserves sequences of the form~\eqref{eqn:exact_sequence_of_duals}.

 The claim now follows from the two facts that $F$ preserves filtered colimits and that filtered colimits in $\ca{D}$ are exact.
\end{proof}

\begin{proof}[Proof of Proposition~\ref{prop:closure_under_affine_functors}.]
 With the notation introduced at the beginning of this section, we have already shown that $\ca{D}$ is pre-geometric and that $FA$ is faithfully flat (see Lemma~\ref{lemma:faithfully_flat_stable}). Thus it only remains to show that $FA$ is an affine algebra. Indeed, the existence of a faithfully flat affine algebra in $\ca{D}$ implies the existence of a faithfully flat covering, hence that $\ca{D}$ is geometric (see Corollary~\ref{cor:faithfully_flat_coverings_vs_algebras}).

 Recall that $FA$ is affine if and only if the functor
\[
\ca{D}(I,U_{FA}-) \colon \ca{D} \rightarrow \Mod_R 
\]
 is faithful and exact, where $U_{FA} \colon \ca{D}_{FA} \rightarrow \ca{D}$ denotes the forgetful functor (see Lemma~\ref{lemma:affine_algebra_characterization}). 

 Let $U_A \colon \ca{C}_A \rightarrow \ca{C}$ denote the corresponding forgetful functor for $\ca{C}$, and let $\overline{U} \colon \ca{D}_{FA} \rightarrow \ca{C}_A$ be the lift of the right adjoint $U$ of $F$. By construction, the diagram
\[
 \xymatrix{\ca{D}_{FA} \ar[r]^-{\overline{U}} \ar[d]_{U_{FA}} & \ca{C}_A \ar[d]^{U_A} \\ \ca{D} \ar[r]_-{U} & \ca{C} }
\]
 is commutative. Thus we get the sequence of natural isomorphisms
\begin{align*}
  \ca{D}(I,U_{FA}-) &\cong \ca{D}(FI,U_{FA}-) \\
& \cong \ca{C}(I,UU_{FA}-)\\
 &= \ca{C}(I,U_A \overline{U}-) \\
&= \ca{C}(I,U_A-) \circ \overline{U} \smash{\rlap{.}}
\end{align*}
 Since $A$ is an affine algebra, we have reduced the problem to showing that $\overline{U}$ is faithful and exact (see Lemma~\ref{lemma:affine_algebra_characterization}). But this is immediate from the faithfulness and exactness of $U$, which in turn follows from the fact that $F$ is cohomologically affine (see Definition~\ref{def:cohomologically_affine}).
\end{proof}

\subsection{Affine morphisms induce affine functors}\label{section:affine_morphism_implies_affine_functor}
 General pullbacks can be computed as pullbacks along the diagonal, and the diagonal of an Adams stack is an affine morphism. In our proof of Theorem~\ref{thm:pushouts}, we use the fact that the induced functor
\[
f^{\ast } \colon  \QCoh(Y) \rightarrow \QCoh(X)
\]
 of an affine morphism $f \colon X \rightarrow Y$ is an affine functor.

 This is well-known in the case of Artin stacks (see \cite[Proposition~3.9]{ALPER}). Since Adams stacks are more general --- while they are assumed to be quasi-compact, they might only have a \emph{flat} cover by a scheme, not subject to any finiteness conditions --- we provide a detailed proof below. To follow the proof of \cite[Proposition~3.9]{ALPER}, we need to show that $f^{\ast}$ satisfies the \emph{flat base change formula}, or \emph{Beck-Chevalley condition} (see Diagram~\ref{eqn:diagram_of_lifts} below), and then use faithfully flat descent (or, in categorical language, comonadicity of faithfully flat coverings of a tensor category) to reduce to the affine case.

\begin{lemma}\label{lemma:composition_coclosed}
 Let $F \colon \ca{C} \rightarrow \ca{D}$ and $ G \colon \ca{D} \rightarrow \ca{E}$ be tensor functors, with right adjoints $U$ and $V$ respectively. Then for all objects $C \in \ca{C}$ and $E \in \ca{E}$ the diagram
\[
\xymatrix{C \otimes UV(E) \ar[rd]_{\chi^{GF}_{C,E}} \ar[rr]^{\chi^F_{C,VE}} && U(FC \otimes VE) \ar[ld]^{U(\chi^{G}_{FC,E})} \\ & UV\bigl( GF(C) \otimes E \bigr) }
% \chi^{GF}_{C,E}=U(\chi^{G}_{FC,E}) \circ \chi^F_{C,VE} \smash{\rlap{,}} 
\]
 is commutative, where $\chi$ is the natural transformation defined in Diagram~\eqref{eqn:projection_formula}.
\end{lemma} 

\begin{proof}
 This follows from a straightforward diagram chase, using the fact that the unit and counit of a strong monoidal left adjoint functor are monoidal natural transformations.
\end{proof}

\begin{proof}[Proof of Proposition~\ref{prop:affine_morphism_implies_affine_functor}.]
 Let $f \colon X \rightarrow Y$ be an affine morphism. We want to show that $f^{\ast}$ is an affine functor. Since $\QCoh(Y)$ is geometric, any cohomologically affine functor is affine (see Proposition~\ref{prop:cohomologically_affine_implies_affine}). In other words, it suffices to show that the right adjoint of $f^{\ast}$ is faithful and exact.
 
 Pick a faithfully flat affine algebra $A \in \QCoh(Y)$, with corresponding faithfully flat covering $Y_0 \rightarrow Y$. By assumption, the pullback of $Y_0$ along $f \colon X \rightarrow Y$ is affine. From Corollary~\ref{cor:pushout_along_affine_functor} and Proposition~\ref{prop:pushout_along_affine} we know that $\QCoh(X \pb{Y} Y_0)$ is equivalent to $\bigl(\QCoh(X)\bigr)_{f^{\ast} A}$. In other words, we have shown that $f^{\ast} A$ is an affine algebra. It is also faithfully flat by Lemma~\ref{lemma:faithfully_flat_stable}.

 The composite
\[
 \xymatrix@!C=55pt{\QCoh(Y) \ar[r]^-{f^{\ast}} & \QCoh(X) \ar[r]^{f^{\ast}A \otimes -} & \Mod_B }
\]
 is a tensor functor with (pre-)geometric domain and affine codomain, so it is affine by Proposition~\ref{prop:affine_target_implies_affine} and hence coclosed by Proposition~\ref{prop:cohomologically_affine_implies_affine}. The same reasoning applies to the functor $f^{\ast} A \otimes -$. Applying Lemma~\ref{lemma:composition_coclosed} to this composite, we find that the morphism
\begin{equation}\label{eqn:chi_is_iso}
 \chi^{f^{\ast}}_{A,VM} \colon A \otimes f_{\ast} VM \rightarrow f_{\ast}(f^{\ast} A \otimes VM) 
\end{equation}
 is invertible for all $M \in \Mod_B$, where $V$ denotes the right adjoint of $f^{\ast} A \otimes -$. Since $f^{\ast}A$ is faithfully flat, the adjunction $f^{\ast} A \otimes - \dashv V$ is comonadic. It follows that every object of $\QCoh(X)$ can be written as equalizer of objects in the image of $V$. Since both the domain and codomain of the morphism~\ref{eqn:chi_is_iso} preserve equalizers, it follows that
\[
 \chi^{f^{\ast}}_{A,N} \colon A \otimes f_{\ast} N \rightarrow  f_{\ast} (f^{\ast}A \otimes N)
\]
 is an isomorphism for all $N \in \QCoh(X)$. This implies that the diagram
\begin{equation}\label{eqn:diagram_of_lifts}
\vcenter{ \xymatrix{ \QCoh(Y)_A  & \QCoh(X)_{f^{\ast} A} \ar[l]_-{\overline{f_{\ast}}} \\ \QCoh(X) \ar[u]^{A\otimes -}  & \QCoh(Y) \ar[l]^{f_{\ast}} \ar[u]_{f^{\ast} A \otimes -}  
}}
\end{equation}
 commutes up to isomorphism. (This is the Beck-Chevalley condition alluded to at the beginning of this section.)

 The functor $\overline{f_{\ast}}$ is right adjoint to the tensor functor $\overline{f^{\ast}}$, which is a tensor functor between two categories of modules over commutative $R$-algebras. Such tensor functors are induced by morphisms of algebras (this is  not hard to see directly, but it also follows from Lemma~\ref{lemma:module_category_universal_property}). Thus the right adjoint is given by the corresponding forgetful functor, so it is in particular faithful and exact.

 To conclude the proof, note that the above reasoning and the natural isomorphism in Diagram~\ref{eqn:diagram_of_lifts} imply that $A \otimes f_{\ast}(-)$ is faithful and cocontinuous. Since $A$ is faithfully flat, it follows that $f_\ast$ is faithful and cocontinuous. This shows that the functor $f^{\ast}$ is indeed cohomologically affine, and therefore affine (see Proposition~\ref{prop:cohomologically_affine_implies_affine}).
\end{proof}

\section{Geometric categories in characteristic zero}\label{section:description}

\subsection{Overview}
 Throughout this section we assume that $R$ is a $\mathbb{Q}$-algebra. In fact, since the degenerate case $R=0$ is easy to deal with, we can assume that $\mathbb{Q} \subseteq R$. The goal of this section is to provide an intrinsic characterization of geometric categories in characteristic zero. To do this, we need to construct a faithfully flat covering, which by Corollary~\ref{cor:faithfully_flat_coverings_vs_algebras} amounts to the construction of a faithfully flat affine algebra $A \in \ca{C}$. We will follow Deligne's proof of the analogous theorem for Tannakian categories (\cite[Th\'eor\`eme~7.1]{DELIGNE}).
 
 By Lemma~\ref{lemma:affine_algebra_characterization}, we need to construct a faithfully flat commutative algebra such that $A$ is projective as an $A$-module, and such that $A$ is a generator of the category of $A$-modules. In fact, we can construct two separate faithfully flat algebras with the respective properties, and then take $A$ to be their tensor product. It turns out that the two constructions that Deligne used in \cite[\S 7]{DELIGNE} work in the context of general pre-geometric categories, though the proofs (especially of faithful flatness) have to be modified. Deligne does not rely on the notion of an affine algebra in $\ca{C}$, so there is also some work involved in checking that the resulting algebra is affine.
 
 The basic idea behind the construction is rather simple: if $\ca{C}$ is geometric, then an epimorphism in $\ca{C}$ whose codomain has a dual is ``locally split,'' that is, its image under a faithfully flat covering $\ca{C} \rightarrow \Mod_B$ is split. Deligne's idea is to construct an algebra $B$ which is universal with the property that a fixed epimorphism $p$ with codomain equal to the unit object is split after tensoring with $B$, and then checking that this algebra is faithfully flat. If we successively split more and more epimorphisms this way, we eventually obtain a faithfully flat algebra which is projective as a module over itself (see Proposition~\ref{prop:projective_algebra}).
 
 To get an algebra that is a generator, we use Deligne's second construction. For each object $X$ with integral dimension (rank), there exists a faithfully flat algebra $A$ such that $A \otimes X \cong A^{\oplus n}$ as $A$-modules (see \cite[Proposition~7.17]{DELIGNE}). Again applying this successively, we obtain an algebra $A$ which has this property for a generating set of objects of $\ca{C}$, and it follows that $A$ is a generator of $\ca{C}_A$.
 
 While Deligne's constructions work for general pre-geometric categories, the proofs have to be modified a bit. Proving that the first algebra is faithfully flat is more involved in the first case (in the situation Deligne considers, every non-zero object is faithfully flat). Luckily, Deligne's proof that the algebra is not zero requires only minor modifications to show that the criterion for faithful flatness of Lemma~\ref{lemma:faithfully_flat_condition} is satisfied. However, to get the construction to work in the first place, we need to impose an additional axiom that is automatically satisfied in the case Deligne considers. 
 
 For the second construction, faithful flatness is not an issue, since Deligne proves that the unit is split. His proof therefore works more or less verbatim in our setting. What causes some difficulties in the general case is that the rank of an object with a dual need not be an integer (for example, the rank of a locally free coherent sheaf is only locally constant). By suitably ``padding'' the object with the correct number of ``free'' summands, we can however show that every such object is a direct summand of an object with integral rank (see Lemma~\ref{lemma:integral_rank}).
 
 \subsection{Proof of Theorem~\ref{thm:description}}
 The goal of this section is to prove Theorem~\ref{thm:intrinsic}, which is a version of Theorem~\ref{thm:description} phrased in the language of geometric categories. We will defer the proofs of the two key propositions to the next two sections. Fix a tensor category $\ca{C}$, and let $X \in \ca{C}$ be an object with a dual. Then we define the \emph{trace} of an endomorphism $f$ of  $X$ to be the composite
 \[
 \xymatrix@!C=45pt{I \ar[r]^-{\mathrm{coev}} & X^{\vee} \otimes X \ar[r]^-{s_{X^{\vee},X}} & X \otimes X^{\vee} \ar[r]^{f\otimes \id_{X^{\vee}}} & X \otimes X^{\vee} \ar[r]^-{\mathrm{ev}} & I}
 \]
 in the endomorphism ring $\End(I)$. The \emph{rank} of $X$ is simply the trace of the identity on $X$ (note that what we call the rank was called the dimension of $X$ in \cite[\S 7]{DELIGNE}). We denote the rank of $X$ by $\rk(X)$. The two main tools for Deligne's construction are exterior and symmetric powers of an object $X$. Since we assume that $R$ contains $\mathbb{Q}$, these are direct summands of the $n$-fold tensor power of $X$. Explicitly, the two elements
 \[
 s  =\sum_{\sigma \in \Sigma_n} \sigma \quad \text{and} \quad a=\sum_{\sigma \in \Sigma_n} \sgn(\sigma) \sigma
 \]
 of the group algebra $\mathbb{Q}[\Sigma_n]$ define idempotents $s \slash n!$ and $a \slash n!$ of $X^{\otimes n}$ (where $\sigma$ acts on $X^{\otimes n}$ via the symmetry of the monoidal structure of $\ca{C}$). The splittings of these two idempotents are called the $n$-th \emph{symmetric} respectively the $n$-th \emph{exterior} power of $X$, and they are denoted by $\Sym^n(X)$ respectively $\Lambda^n (X)$. To disambiguate we will denote the corresponding constructions in the category $\ca{C}_A$ of $A$-modules for some commutative algebra $A \in \ca{C}$ by $\Sym^n_A$ and $\Lambda^n_A$. 
 
\begin{rmk}\label{rmk:free_algebras}
 The direct sum
 \[
\Sym(X) \defl \bigoplus_{n \in \mathbb{N}} \Sym^n(X) 
 \]
 is the (underlying object of the) free commutative algebra on $X$, and the $\mathbb{N}$-graded object
 \[
\Lambda(X) \defl \bigr (\Lambda^n(X)\bigl)_{n \in \mathbb{N}} 
 \]
 is the free commutative algebra on $X$ (thought of as an $\mathbb{N}$-graded object concentrated in degree 1) in the category of $\mathbb{N}$-graded objects in $\ca{C}$. The symmetry in this category is defined according to the usual sign conventions (that is, switching objects of odd degrees introduces a factor of $-1$). 
\end{rmk}
 
 \begin{thm}\label{thm:intrinsic}
 Suppose that $R$ is a $\mathbb{Q}$-algebra. Let $\ca{C}$ be an abelian tensor category which has the following properties.
 
 \begin{enumerate}
\item[(i)] There exists a generating set of objects with duals (in other words, the tensor category $\ca{C}$ is pre-geometric).
 
\item[(ii)] If $X \in \ca{C} $ has a dual and if $\rk(X)=0$, then $X=0$.

\item[(iii)] If $X \in \ca{C}$ has a dual, then there exists an integer $n \geq 0$ such that $\Lambda^n X=0$.
 
 \item[(iv)] If $X \in \ca{C}$ has a dual and $p \colon X \rightarrow I$ is an epimorphism, then the dual morphism $I \cong I^{\vee} \rightarrow X^{\vee}$ is a monomorphism whose cokernel has a dual.
 \end{enumerate}
 Then $\ca{C}$ is a geometric category. In particular, there exists an Adams stack $Y$ over $R$ such that $\ca{C} \simeq \QCoh(Y)$.
 \end{thm}

\begin{rmk}
 As we will see in the proof, it suffices that Condition~(iv) holds for objects $X$ constituting a generating set of $\ca{C}$, and that the cokernel in question is flat. However, as we have seen in the introduction, the stronger condition is a necessary condition.
\end{rmk}

\begin{rmk}
 Deligne considered the case where every finitely presentable object of the tensor category $\ca{C}$ has a dual, and where the endomorphism ring of the unit object is a field. These assumptions and Condition~(iii) imply the other conditions. It would be interesting to know whether or not conditions (i)-(iv) are independent in the general case. 

 As a referee pointed out, Conditions~(i) and (iv) alone do not imply (ii): in the category of $\mathbb{Z}$-graded abelian groups (with symmetry involving minus signs), the rank is given by the alternating sum of the ranks of the individual pieces. This pre-geometric tensor category clearly satisfies Condition (iv) since the unit object is projective.
 \end{rmk}
 
 Note that Theorem~\ref{thm:intrinsic}, which is phrased using the language introduced in \S \ref{section:geometric}, immediately implies the corresponding description result for weakly Tannakian categories.

\begin{proof}[Proof of Theorem~\ref{thm:description}]
 By Proposition~\ref{prop:geometric_equals_weakly_Tannakian}, passage to ind-objects (respectively restriction to finitely presentable objects) gives an equivalence between weakly Tannakian categories and geometric categories. Under this equivalence, the conditions of Theorem~\ref{thm:intrinsic} and Theorem~\ref{thm:description} correspond to each other.
\end{proof}

 \begin{rmk}
  Our assumption that $R$ is a non-zero $\mathbb{Q}$-algebra means that for any non-zero symmetric monoidal $R$-linear category, the endomorphism ring $\End(I)$ of the unit object contains $\mathbb{Q}$. This follows from the fact that any non-zero ring homomorphism whose domain is a field is injective. Theorem~\ref{thm:intrinsic} is true in the case $\ca{C}=0$ (with $Y=\Spec(0)$), so we may assume without loss of generality that the tensor category $\ca{C}$ is non-zero, and hence that $\mathbb{Q} \subseteq \End(I)$.
 \end{rmk}

 As already mentioned in the overview of this section, our proof of Theorem~\ref{thm:intrinsic} relies on the same two constructions as Deligne's proof of the analogous result for Tannakian categories in \cite[\S 7]{DELIGNE}. The first key result is \cite[Lemme~7.14]{DELIGNE}, and we will prove its generalization in \S \ref{section:splitting_epis_proof}.
 
 \begin{prop}\label{prop:splitting_epis}
 Suppose that $R$ is a $\mathbb{Q}$-algebra. Let $\ca{C}$ be an abelian tensor category which satisfies Conditions~(i) and (iv) of Theorem~\ref{thm:intrinsic}, let $X$ be an object with a dual, and let $p \colon X \rightarrow I$ be an epimorphism. Then there exists a faithfully flat commutative algebra $A \in \ca{C}$ such that $A \otimes p$ is split as a morphism of $A$-modules.
 \end{prop}
 
 We will use this result to show that there exists a faithfully flat algebra in $\ca{C}$ which is projective as a module over itself (see Proposition~\ref{prop:projective_algebra} below). To do this we need the following lemma.
 
 \begin{lemma}\label{lemma:faithfully_flat_stable_under_filtered_colimits}
 Let $\ca{C}$ be a pre-geometric category. Then faithfully flat commutative algebras in $\ca{C}$ are closed under filtered colimits.
 \end{lemma}

\begin{proof}
 Since filtered colimits are exact in a locally finitely presentable category, flat objects are closed under filtered colimits. The claim follows from the fact that a commutative algebra $A$ is faithfully flat if and only if the sequence
 \[
 \xymatrix{0 \ar[r] & I \ar[r] & A \ar[r] & A\slash I \ar[r] & 0} 
 \]
 consists of flat objects and is exact (see Lemma~\ref{lemma:faithfully_flat_condition}).
\end{proof} 
 
 \begin{prop}\label{prop:projective_algebra}
Suppose that $R$ is a $\mathbb{Q}$-algebra, and let $\ca{C}$ be an abelian tensor category which satisfies Conditions~(i) and (iv) of Theorem~\ref{thm:intrinsic}. Then there exists a faithfully flat algebra $A \in \ca{C}$ such that $A \in \ca{C}_A$ is projective.
 \end{prop}
 
\begin{proof}
 Since there is only a set of isomorphism classes of finitely presentable objects in $\ca{C}$, there is also only a set of isomorphism classes of epimorphisms $p \colon X \rightarrow I$ where $X$ has a dual. From Proposition~\ref{prop:splitting_epis} we know that there is a set of faithfully flat commutative algebras such that each of these epimorphisms is split after base change to one of them. By taking the filtered colimit of all finite tensor products of these algebras, we obtain a single commutative algebra $A$ with the property that $A \otimes p$ is split as a morphism of $A$-modules for \emph{all} such epimorphisms. This algebra is faithfully flat by Lemma~\ref{lemma:faithfully_flat_stable_under_filtered_colimits}, and we claim that $A$ is projective as an $A$-module.
 
 To see this, we need to show that $\ca{C}_A(A,-) \cong \ca{C}(I,U-)$ is right exact, where $U \colon \ca{C}_A \rightarrow \ca{C}$ denotes the forgetful functor. Let $q \colon M \rightarrow N$ be an epimorphism in $\ca{C}_A$. We need to show that for every morphism $I \rightarrow UN$ in $\ca{C}$, there exists a lift to $UM$. We claim that there exists an object $X$ with a dual, together with an epimorphism $p \colon X \rightarrow I$ and a morphism $X \rightarrow UM$ in $\ca{C}$ such that the diagram
 \[
 \xymatrix{X \ar[d]_p \ar[r] & UM \ar[d]^{Uq} \\ I \ar[r] & UN}
 \]
 is commutative. We can prove this claim as follows.
 
 First note that we can find a commutative diagram as above where instead of $X$ we have a (possibly infinite) direct sum $Z$ of objects with duals. Indeed, the pullback of $Uq$ to $I$ is still an epimorphism, and we can write this pullback as a quotient of such a direct sum of objects with duals by Condition~(i) of Theorem~\ref{thm:intrinsic}. It remains to show that there exists a finite direct sum $X \subseteq Z$ such that the restriction of $Z \rightarrow I$ to $X$ is still an epimorphism. To see this, we first write $Z$ as directed union of the finite direct sums $Z_i$ of its summands. Let $Z_i \rightarrow I_i \rightarrow I$ be the image factorization of the restriction of the epimorphism $Z \rightarrow I$ to $Z_i$. Since $\ca{C}$ is Grothendieck abelian, the filtered colimit of the $Z_i \rightarrow I_i$ provides an image factorization of this epimorphism, hence we must have $I=\colim_i I_i$. The fact that $I$ is finitely presentable (see Definition~\ref{def:tensor_category}) implies that $I=I_i$ for some index $i$. Thus $X=Z_i$ gives the desired object with dual.

 Since $A \otimes -$ is left adjoint to $U$, the above diagram is equal to the outer rectangle in the diagram
 \[
 \xymatrix{ X \ar[r]^-{\eta \otimes X} \ar[d]_p & U(A \otimes X) \ar[d]^{U(A\otimes p)} \ar[r] & U(A \otimes UM) \ar[d]^{U(A \otimes U q)} \ar[r]^-{U \varepsilon_M} & UM \ar[d]^{U q} \\ I \ar[r]_-{\eta \otimes I} & U(A \otimes I) \ar[r] & U(A \otimes UN) \ar[r]_-{U \varepsilon_N} & UN}
 \]
 where $\eta$ denotes the unit of $A$. Thus it suffices to show that there exists a lift
 \[
 \xymatrix{& U(A \otimes X) \ar[d]^{U(A \otimes p)} \\ I \ar@{-->}[ru] \ar[r]_-{\eta \otimes I} & U(A \otimes I)}
 \]
 in $\ca{C}$. Again using the adjunction $A \otimes - \dashv U$, this corresponds to a splitting of $A \otimes p$ as a morphism of $A$-modules, which exists by construction of $A$.
\end{proof}

 To show that there exists a faithfully flat commutative algebra $A$ in $\ca{C}$ with the property that $A$ is a generator of the category $\ca{C}_A$ of $A$-modules, we use the following construction of Deligne. In this case, the proof from \cite[\S 7]{DELIGNE} works almost verbatim. We prove the following proposition in \S \ref{section:generator}.
 
 \begin{prop}\label{prop:unit_direct_summand}
 Suppose that $R$ is a $\mathbb{Q}$-algebra. Let $\ca{C}$ be an abelian tensor category satisfying Conditions~(i)-(iii) of Theorem~\ref{thm:intrinsic}. Let $X$ be an object with a dual in $\ca{C}$, and let $M$ be a direct summand of $A \otimes X$ in the category $\ca{C}_A$ of $A$-modules. If $\rk(M)$ is a natural number greater than zero, there exists a faithfully flat $A$-algebra $B$ such that $B$ is a direct summand of $B \otimes_A M$ as a $B$-module.
 \end{prop}
 
 In order to show that the above proposition implies the existence of the desired kind of faithfully flat algebra, we first need the following two results. First note that $\Lambda^n(X)$ and $\Sym^n(X)$ have a dual if $X$ does, since both are direct summands of $X^{\otimes n}$.
 
\begin{lemma}\label{lemma:trace_of_exterior_power}
 Let $\ca{C}$ be a tensor category and let $X \in \ca{C}$ be an object with a dual. Then the formulas
\begin{equation}\label{eqn:trace_of_exterior_power}
 \rk\bigr(\Lambda^n (X)\bigl) = \rk(X) \bigr(\rk(X) - 1 \bigl) \bigr(\rk(X) - 2 \bigl) \cdots \bigr(\rk(X) - n + 1 \bigl) \slash n! 
\end{equation}
and
\begin{equation}
 \rk \bigr(\Sym^n (X) \bigl) = \bigl( \rk(X) +n-1 \bigr) \bigl( \rk(X) +n-2 \bigr) \cdots \rk(X) \slash n!
\end{equation}
 hold in the ring of endomorphisms $\End(I)$ of the unit object $I$ of $\ca{C}$.
\end{lemma}

\begin{proof}
 This is proved in \cite[Formula~7.1.2]{DELIGNE}. To see that it works at the desired level of generality, we give a brief summary. First note that by cyclicity of traces, the trace of the identity of $\Lambda^n(X)$ is equal to the trace of the idempotent $a \slash n!$ of $X^{\otimes n}$. Similarly, the trace of the identity of $\Sym^n(X)$ is equal to the trace of the idempotent $s \slash n!$.

 Either as in the proof of \cite[Lemme~7.2]{DELIGNE} or, perhaps more straightforwardly, using string diagrams, one checks that the trace of a cyclic permutation of $X^{\otimes k}$ is equal to the trace of the identity of $X$, that is, to the rank $\rk(X)$ of $X$. Since the trace sends tensor products of morphisms to composites (this is again readily checked using string diagrams), it follows that the traces of the two idempotents
\[
 a \slash n!= \frac{1}{n!} \sum_{\sigma \in \Sigma_n} \sgn(\sigma) \sigma \quad \text{and} \quad s \slash n!= \frac{1}{n!}  \sum_{\sigma \in \Sigma_n} \sigma
\]
 are polynomials in $\rk(X)$ with natural number coefficients (independent of the $\mathbb{Q}$-linear tensor category $\ca{C}$). The claim follows from applying this fact in the case where $\ca{C}$ is the tensor category of $\mathbb{Q}$-vector spaces.
\end{proof}

 \begin{lemma}\label{lemma:integral_rank}
 Let $\ca{C}$ be an abelian tensor category which satisfies Condition~(iii) of Theorem~\ref{thm:intrinsic}, and let $X \in \ca{C}$ be an object with a dual. Then there exists an object $Y \in \ca{C}$ with a dual such that $X$ is a direct summand of $Y$ and $\rk(Y)$ is a natural number.
 \end{lemma}
 
 \begin{proof}
 By Condition~(iii) of Theorem~\ref{thm:intrinsic}, there exists a natural number $\ell$ such that $\Lambda^{\ell}(X)=0$. Without loss of generality we can assume $\ell >1$. From Lemma~\ref{lemma:trace_of_exterior_power} it follows that the equation
\begin{equation}\label{eqn:vanishing_polynomial}
 \rk(X) \bigr(\rk(X) - 1 \bigl) \bigr(\rk(X) - 2 \bigl) \cdots \bigr(\rk(X) - \ell + 1 \bigl)=0
\end{equation}
 holds in the ring of endomorphisms $\End(I)$ of the unit object. Since this ring can have zero divisors, we cannot conclude that the rank of $X$ is equal to a natural number. A guiding example to keep in mind is the case where $\ca{C}$ is a finite product of copies of the tensor category of $\mathbb{Q}$-vector spaces (respectively $R$-modules). In that case, the rank is given by a vector with natural number entries. We claim that this is the case for a general tensor category $\ca{C}$ as well.

 Consider the ring homomorphism
\[
 \varphi \colon \mathbb{Q}[t] \rightarrow \End(I)
\]
 which sends $t$ to $\rk(X)$ and let $A=\mathbb{Q}[t] \slash \ker(\varphi)$. Since the polynomial of Formula~\eqref{eqn:vanishing_polynomial} lies in the kernel of $\varphi$, it follows that this kernel is generated by a polynomial of the form $(t-r_1)\cdots (t-r_n)$, where all the $r_k$ are \emph{distinct} natural numbers. Thus $A$ is isomorphic to a finite product of copies of $\mathbb{Q}$, and under this isomorphism, the image of $t$ in $A$ is given by $(r_1, \ldots, r_n)$.

 The images $e_k \in \End(I)$ of the orthogonal idempotents in $A$ define a direct sum decomposition
\[
 I \cong \bigoplus_{k=1}^n I_k
\]
 of the unit object of $I$. By construction, the rank of $I_k$ is equal to the idempotent $e_k$ in $\End(I)$. Writing $X_k$ for $X \otimes I_k$ we have
 $\rk(X_k)=r_k e_k$. By adding a suitable number of copies of $I_k$, we get an object $Y_k$ which contains $X_k$ as a direct summand and such that $\rk(X_k)=r e_k$, where $r$ is the maximum of the $r_k$. The object
\[
 Y = \bigoplus_{k=1}^n Y_k
\]
 therefore has rank $r \in \mathbb{N}$, and it contains a copy of $X$ as a direct summand.
 \end{proof}

\begin{prop}\label{prop:generating_algebra}
 Suppose that $R$ is a $\mathbb{Q}$-algebra. Let $\ca{C}$ be an abelian tensor category satisfying Conditions~(i)-(iii) of Theorem~\ref{thm:intrinsic}. Then there exists a faithfully flat commutative algebra $A$ such that $A$ is a generator of the category $\ca{C}_A$ of $A$-modules.
\end{prop}

\begin{proof}
 We first claim that for each object $X \in \ca{C}$ with dual such that $n=\rk(X)$ is a positive integer, there exists a faithfully flat algebra $A$ with the property that $A\otimes X \cong A^{\oplus n}$ as $A$-modules. 

 Applying Proposition~\ref{prop:unit_direct_summand} $n$ times, we find that there exists a faithfully flat algebra $A$ and an $A$-module $N$ such that $A \otimes X \cong A^{\oplus n} \oplus N$, so it suffices to show that $N=0$. We can use Deligne's argument from \cite[\S 7]{DELIGNE}. Namely, using the fact that $\Lambda_A(-)$ is the free graded commutative algebra functor (see Remark~\ref{rmk:free_algebras}), we find that there is an isomorphism
\[
 \Lambda_A^{n+1}(A^{\oplus n} \oplus N) \cong \bigoplus_{i+j=n+1} \Lambda_A^i(A^{\oplus n}) \otimes_A \Lambda^j_A (N)
\]
 of $A$-modules. Since $\Lambda^n_A(A^{\oplus n})=A$, this shows that $\Lambda^1(N)=N$ is a direct summand of $\Lambda^{n+1}_A(A^{\oplus n} \oplus N)$, so it suffices to show that
\[
 \Lambda^{n+1}_A(A^{\oplus n} \oplus N) \cong  \Lambda^{n+1}_A(A \otimes X) \cong A \otimes \bigr(\Lambda^{n+1}(X) \bigl) 
\]
 is equal to zero. This follows from the fact that the rank of $\Lambda^{n+1}(X)$ is zero (see Formula~\eqref{eqn:trace_of_exterior_power}) and Condition~(ii) of Theorem~\ref{thm:intrinsic}. This concludes the proof of the first claim.

 By assumption, there is a generating set of $\ca{C}$ consisting of objects with duals (Condition~(i) of Theorem~\ref{thm:intrinsic}). By Lemma~\ref{lemma:integral_rank}, there exists a generating set of objects which also have integral rank. By taking a filtered colimit of finite tensor products, we obtain a commutative algebra $A$ with the property that $A \otimes X \cong A^{\oplus n}$ for each object $X$ of a generating set (with rank $n$ depending on $X$). This algebra is faithfully flat by Lemma~\ref{lemma:faithfully_flat_stable_under_filtered_colimits}, and $A$ is a generator of $\ca{C}_A$ since every $A$-module can be written as a quotient of  a free module.
\end{proof}

 \begin{proof}[Proof of Theorem~\ref{thm:intrinsic}.]
 Let $\ca{C}$ be a tensor categories satisfying Conditions~(i)-(iv) of Theorem~\ref{thm:intrinsic}. By Proposition~\ref{prop:projective_algebra}, there exists a faithfully flat algebra $A \in \ca{C}$ which is projective as an $A$-module, and by Proposition~\ref{prop:generating_algebra} there exists a faithfully flat algebra $A^{\prime}$ which is a generator of the category of $A^{\prime}$-modules. 

 Note that if the unit of a tensor category $\ca{D}$ is projective (or a generator), the same is true for all tensor categories of modules over commutative algebras in $\ca{D}$. Applying this reasoning to $\ca{D}=\ca{C}_A$ and $\ca{D}=\ca{C}_{A^{\prime}}$, we find that $A \otimes A^{\prime}$ is a projective generator of $\ca{C}_{A \otimes A^{\prime}}$. Thus $A \otimes A^{\prime}$ is a faithfully flat affine algebra (see Definition~\ref{def:affine_algebra}). This concludes the proof that $\ca{C}$ is geometric, with faithfully flat covering given by $\ca{C} \rightarrow \ca{C}_{A\otimes A^{\prime}}$ (see Corollary~\ref{cor:faithfully_flat_coverings_vs_algebras}).
 \end{proof}

\subsection{Constructing an algebra which is projective as a module}\label{section:splitting_epis_proof}
 In this section we prove that Deligne's construction of an algebra satisfying the conclusions of Proposition~\ref{prop:projective_algebra} still works under the weaker assumptions of Theorem~\ref{thm:intrinsic}. Deligne's proof that the resulting algebra is faithfully flat has to be modified, though. We will need the following lemmas.
 
 \begin{lemma}\label{lemma:pushout_of_algebras}
  Let $\ca{C}$ be a tensor category, $B \in \ca{C}$ a commutative algebra, and let $f \colon  I \rightarrow B$ be a morphism in $\ca{C}$. Let $\hat{f} \colon B \rightarrow B$ be the corresponding morphism of $B$-modules under the free $B$-module adjunction. Let %Then the underlying morphism of $p$ in the pushout diagram
  \[
  \xymatrix{\Sym(I) \ar[r] \ar[d] & B \ar[d]^{p}\\  I \ar[r] & I \ten{\Sym(I)} B}
  \]
 be the pushout in the category of commutative algebras in $\ca{C}$, where the two morphisms with domain $\Sym(I)$ correspond to $f$ and $\id_I$ under the adjunction of Remark~\ref{rmk:free_algebras}. Then the morphism of $B$-modules underlying $p$ is isomorphic to the cokernel of the morphism
  \[
 \hat{f}-\id_B \colon B \rightarrow B
  \]
 in $\ca{C}_B$.
 \end{lemma}

\begin{proof}
 Pushouts of commutative algebras are given by tensor products, and $I$ is the initial commutative algebra. Thus $p$ is given by the coequalizer of the diagram
 \[
 \xymatrix{I \otimes \Sym(I) \otimes B \ar@<0.5ex>[r] \ar@<-0.5ex>[r] & I \otimes B}
 \]
 in $\ca{C}$. Using the fact that $\Sym(I)$ is a countable direct sum of copies of $I$ and that the adjunction $B \otimes - \colon \ca{C} \rightleftarrows \ca{C}_B \colon U $ is monoidal we find that a morphism coequalizes the above diagram if and only if it coequalizes the diagram
 \[
 \xymatrix{B \ar@<0.5ex>[r]^{\hat{f}^{\otimes_B n}} \ar@<-0.5ex>[r]_{\id_B} & B}
 \]
 for all $n \in \mathbb{N}$. Since tensor products of endomorphisms of the unit object $B \in \ca{C}_B$ coincide with compositions (by the Eckmann-Hilton argument), we have $\hat{f}^{\otimes_B n} = \hat{f} \circ \ldots \circ \hat{f}$. Thus a morphism simultaneously coequalizes the above diagrams for all $n$ if and only if it coequalizes the diagram for $n=1$.
\end{proof}

We will apply this lemma in the case where the target is equal to a graded algebra. Then the computation of the pushout can be further simplified.

\begin{lemma}\label{lemma:graded_pushout}
 Suppose that in the situation of Lemma~\ref{lemma:pushout_of_algebras}, the algebra
 \[
 B = \bigoplus_{n \in \mathbb{N}} B_n
 \]
  is $\mathbb{N}$-graded\footnote{Here we are using the symmetry on graded objects that does \emph{not} introduce signs.}, and that $f \colon I \rightarrow B$ factors through one of the $B_n$. Then the underlying object of the algebra $I \ten{\Sym(I)} B$ is given by the direct sum
  \[
 \bigoplus_{i=0}^{n-1}  \colim_{k\in \mathbb{N}} B_{i+kn} \smash{\rlap{,}}
  \]
  where the morphisms $B_{i+kn} \rightarrow B_{i+(k+1)n}$ are given by the composite
\begin{equation}\label{eqn:connecting_morphism}
 \xymatrix{B_{i+kn} \ar[r]^-{\cong} & I \otimes B_{i+kn} \ar[r]^-{f\otimes \id} & B_n \otimes B_{i+kn} \ar[r]^-{\mu} & B_{i+(k+1)n}}
\end{equation} 
  in $\ca{C}$. In particular, if $B$ is flat, then so is $I \ten{\Sym(I)} B$.
\end{lemma}

\begin{proof}
 With the notation of Lemma~\ref{lemma:pushout_of_algebras}, the restriction of the morphism $\hat{f}$ to $B_{i+kn}$ is given by Formula~\eqref{eqn:connecting_morphism}. A morphism $g \colon B \rightarrow C$ (which amounts to a collection of morphisms $g_j \colon B_j \rightarrow C$) coequalizes $\hat{f}$ and $\id_B$ if and only if the triangle
 \[
 \xymatrix{B_{i+kn} \ar[rd]_{g_{i+kn}} \ar[rr] && B_{i+(k+1)n} \ar[ld]^{g_{i+(k+1)n}}\\ & C}
 \]
 is commutative for all $k$ and all $i=0, \ldots, n-1$. 
 
 Flatness is immediate from the fact that finite direct sums and filtered colimits of flat objects in $\ca{C}$ are flat.
\end{proof}

 Using these facts we can show that Proposition~\ref{prop:splitting_epis} is true under our weakened assumptions.
 
\begin{proof}[Proof of Proposition~\ref{prop:splitting_epis}.]
 Let $p \colon X \rightarrow I$ be an epimorphism whose domain $X$ has a dual. We need to construct a faithfully flat algebra $A$ such that $A \otimes p$ is split as a morphism of $A$-modules.
 
 Let $p^{\vee} \colon I \rightarrow X^{\vee}$ be the dual morphism of $p$. Deligne showed that --- in the case where $\ca{C}_{\fp}$ is a rigid abelian category --- the pushout $A=I \ten{\Sym(I)} \Sym(X^{\vee})$ of the diagram
 \[
 \xymatrix{ \Sym(I) \ar[d] \ar[r]^{\Sym(p^{\vee})} & \Sym(X^{\vee}) \ar[d] \\ 
 I \ar[r] & I \ten{\Sym(I)} \Sym(X^{\vee})}
 \]
 has the desired properties. We claim that this is still true under our weakened hypotheses.
 
 First note that $A \otimes p$ is split as a morphism of $A$-modules: by construction, the diagram
 \[
 \xymatrix{ X^{\vee} \ar[r] &  \Sym(X^{\vee}) \ar[d] \\ I \ar[u]^{p^{\vee}} \ar[r]_-{\eta} & A}
 \]
 in $\ca{C}$ is commutative. Applying the free $A$-module functor to the left hand side, we get a commutative triangle
 \[
 \xymatrix{& A \otimes X^{\vee} \ar[rd] \\ A\otimes I \ar[rr]_-{\cong} \ar[ru]^{A \otimes p^{\vee}} && A }
 \]
 in the category $\ca{C}_A$ of $A$-modules. This shows that the dual $A \otimes p^{\vee}$ of $A\otimes p$ has a retraction, and we obtain the desired section of $A\otimes p$ by taking the dual of this retraction.
 
 It remains to check that the algebra $A$ is faithfully flat. We will check that the conditions of Lemma~\ref{lemma:faithfully_flat_condition} are satisfied, that is, that the sequence
 \[
 \xymatrix{0 \ar[r] & I \ar[r]^{\eta} & A \ar[r] & A \slash I \ar[r] & 0}
 \]
 is exact and that $A \slash I$ is flat. 
 
 From Lemma~\ref{lemma:graded_pushout} and Formula~\eqref{eqn:connecting_morphism} it follows that $A$ is flat and that the unit of $A$ is given by the filtered colimit of the diagram
 \[
\xymatrix{ I \ar[r] \ar[d] & I \ar[r] \ar[d] & I \ar[r] \ar[d] & \ldots \ar[d] \\ I \ar[r] & \Sym^1(X^{\vee}) \ar[r] & \Sym^2(X^{\vee}) \ar[r] & \ldots }
 \]
 where the vertical morphisms $I \rightarrow \Sym^{n}(X^{\vee})$ are given by the $n$-fold product of $p^{\vee}$. Since monomorphisms and flat objects are stable under filtered colimits, it suffices to show that each vertical morphism in the above diagram is monic and has a flat cokernel.
 
 Note that $(p^{\vee})^{\otimes n} \colon I \rightarrow (X^{\vee})^{\otimes n}$ commutes with the idempotent $s \slash n!$ on the codomain whose splitting is $\Sym^n(X^{\vee})$. This reduces the problem to showing that $(p^{\vee})^{\otimes n}$ is a monic with flat cokernel. Recall that Condition~(iv) of Theorem~\ref{thm:intrinsic} states that $p^{\vee}$ is monic with dualizable (hence flat) cokernel. Since tensoring with $X^{\vee}$ preserves monomorphisms and flat objects, we find that the successive quotients in the sequence
 \[
 \xymatrix@!C=50pt{I \ar[r]^{p^{\vee}} & X^{\vee} \ar[r]^-{p^{\vee} \otimes X^{\vee}} & (X^{\vee})^{\otimes 2} \ar[r]^-{p^{\vee} \otimes (X^{\vee})^{\otimes 2}} & \ldots \ar[r]^-{p^{\vee} \otimes (X^{\vee})^{\otimes n-1}} & (X^{\vee})^{\otimes n}}
 \]
 of monomorphisms are flat. The conclusion therefore follows from Lemma~\ref{lemma:flatness_from_filtration} below.
 \end{proof}

 \begin{lemma}\label{lemma:flatness_from_filtration}
  Let $\ca{C}$ be a pre-geometric tensor category. Then flat objects are stable under extensions. Moreover, if a monomorphism $X \rightarrow Y$ can be written as a composite
  \[
 \xymatrix{ X=F_0 \ar[r] & F_1 \ar[r] & \ldots \ar[r] & F_n=Y}
  \]
 of monomorphisms with the property that the successive quotients $F_i \slash F_{i-1}$ are flat, then $Y \slash X$ is flat.
 \end{lemma}
 
 \begin{proof}
 Recall that there are enough flat objects in a pre-geometric category to set up a theory of Tor-functors (see Remark~\ref{rmk:flat_resolutions_hence_tame}). Thus flatness is equivalent to the vanishing of Tor-functors, and the first claim follows from the long exact sequence for Tor-functors. More explicitly, we can argue as follows. We first claim that a short exact sequence
\[
 \xymatrix{0 \ar[r] & A \ar[r] & B \ar[r] & C \ar[r] & 0 }
\] 
 with $C$ flat is preserved by tensoring with any other object $M$ in $\ca{C}$. To see this, let
 \[
 \ldots \rightarrow M_i \rightarrow \ldots \rightarrow M_0 \rightarrow M
 \]
 be a resolution of $M$ such that each $M_i$ is a (possibly infinite) direct sum of objects with duals. Since $\ca{C}$ is Grothendieck abelian, we get an exact sequence
 \[
 \xymatrix{0 \ar[r] & A \otimes M_\ast \ar[r] & B \otimes M_\ast \ar[r] & C \otimes M_\ast \ar[r]  & 0 }
 \]
 of chain complexes in $\ca{C}$. The resulting long exact sequence in homology and the exactness of $C \otimes M_\ast$ above degree zero imply that
 \[
 \xymatrix{0 \ar[r] & A \otimes M \ar[r] & B \otimes M \ar[r] & C \otimes M \ar[r]  & 0 }
 \]
 is an exact sequence in $\ca{C}$.
 
 Now let $i \colon M \rightarrow N$ be a monomorphism, and assume that both $A$ and $C$ are flat. The snake lemma applied to the diagram
 \[
 \xymatrix{0 \ar[r] & A \otimes M \ar[d] \ar[r] & B \otimes M \ar[r] \ar[d]  & C \otimes M \ar[d] \ar[r]  & 0 \\
0 \ar[r] & A \otimes N \ar[r] & B \otimes N \ar[r] & C \otimes N \ar[r]  & 0    }
 \]
 implies that $B$ is flat as well.
 
 The second claim of the lemma now follows by downward induction. The object $F_n\slash F_{n-1}$ is flat by assumption, and if $F_n \slash F_{n-i}$ is flat, then so is $F_n \slash F_{n-i-1}$. Indeed, the sequence
 \[
 \xymatrix{0 \ar[r] & F_{n-i} \slash F_{n-i-1} \ar[r] & F_n \slash F_{n-i-1} \ar[r] & F_n \slash F_{n-i} \ar[r] & 0}
 \]
 is exact, $F_{n-i} \slash F_{n-i-1}$ is flat by assumption, and we just argued that flat objects in $\ca{C}$ are closed under extensions.
 \end{proof}

\subsection{Constructing an algebra which is a generator}\label{section:generator}

In order to prove Proposition~\ref{prop:unit_direct_summand}, we again use Deligne's construction (see \cite[Lemme~7.15]{DELIGNE}). In this case, both the construction and the proof Deligne gives work at the required level of generality. To keep the account self contained, we give an outline of the proof here.

\begin{proof}[Proof of Proposition~\ref{prop:unit_direct_summand}.]
Let $X \in \ca{C}$ be an object with a dual, and let $M$ be a direct summand of $A\otimes X$ in the category $\ca{C}_A$ of $A$-modules whose rank $\rk(M)$ is a natural number greater than zero. Our goal is to construct a faithfully flat commutative $A$-algbera $B$ such that $B$ is a direct summand of $B\otimes_A M$.

 To give such a direct summand amounts to giving a pair of morphisms
 \[
 \xymatrix{B \ar[r]^-{u} & B\otimes_A M \ar[r]^-{v} & B}
 \]
 of $B$-modules whose composite is the identity. 
 
 As in the proof of Proposition~\ref{prop:projective_algebra}, we claim that the universal algebra for which such a pair of morphisms exist is faithfully flat. First note that the universal algebra in question is given by the pushout $B=A \ten{\Sym_A(A)} \Sym_A(M \oplus M^{\vee})$ of the diagram
 \[
 \xymatrix{\Sym_A(A) \ar[r] \ar[d] & \Sym_A(M \oplus M^{\vee}) \ar[d] \\ A \ar[r] & A \ten{\Sym_A(A)} \Sym_A(M \oplus M^{\vee})}
 \]
 in the category of commutative $A$-algebras, where the unlabeled arrows correspond to $\id_A$ and the morphism
 \[
\xymatrix{  A \ar[r]^-{\coev} & M^{\vee} \ten{A} M  \ar[r]^-{s_{M^{\vee},M}} & M \ten{A} M^{\vee}\; \ar@{>->}[r] & \Sym_A(M) \ten{A} \Sym_A(M^{\vee}) }
 \]
 of degree two under the free commutative algebra adjunction. Indeed, to give a morphism $B \rightarrow C$ of commutative algebras amounts to giving morphisms $v \colon M \rightarrow C$ and $\overline{u} \colon M^{\vee} \rightarrow C$ such that the composite
 \[
 \xymatrix{A \ar[r]^-{\coev} & M^{\vee} \ten{A} M \ar[r]^-{\overline{u} \ten{A} v} & C \ten{A} C \ar[r]^-{\mu} & C}
 \]
 is equal to the unit $\eta \colon A \rightarrow C$. After applying the free $C$-module functor and taking the dual morphism of $\overline{u}$ we get the desired pair of morphisms that exhibit $C$ as direct summand of $C\otimes_A M$. We can in particular apply this reasoning to the identity on $B$.
 
 It remains to show that $B$ is a faithfully flat algebra. From Lemma~\ref{lemma:graded_pushout} (applied to the tensor category $\ca{C}_A$) we know that $B$ is a flat algebra. In \cite[Lemme~7.16]{DELIGNE}, Deligne showed that the unit $A \rightarrow B$ is split as a morphism of $A$-modules. His construction of a retraction still works under our weakened assumptions. Since a flat algebra with split unit is faithfully flat, this concludes the proof.

 To convince the reader that the construction indeed works under our weakened assumptions, we give a detailed account of Deligne's construction in Lemma~\ref{lemma:unit_has_section} below. 
\end{proof}

\begin{lemma}\label{lemma:unit_has_section}
 The unit $A \rightarrow A \ten{\Sym_A(A)} \Sym_A(M \oplus M^{\vee})$ is split as a morphism of $A$-modules.
\end{lemma}

\begin{proof}
 From Lemma~\ref{lemma:graded_pushout} we know that $B \defl A \ten{\Sym_A(A)} \Sym(M \oplus M^{\vee})$ can be written as a direct sum of two $A$-modules obtained from the even and odd degrees of the algebra $\Sym_A(M \oplus M^{\vee})$. The unit factors through the even degree part, so we simply set our retraction to be zero on the odd degree part. 

 Using the isomorphism $\Sym_A(M \oplus M^{\vee}) \cong \Sym_A(M) \ten{A} \Sym_A(M^{\vee})$ and Formula~\eqref{eqn:connecting_morphism}, we find that it suffices to construct a retraction of the filtered colimit of the morphisms
\begin{equation}\label{eqn:multiplication_with_delta}
 \delta \colon  \Sym^n_A(M) \ten{A} \Sym^n_A(M^{\vee}) \rightarrow  \Sym^{n+1}_A(M) \ten{A} \Sym^{n+1}_A(M^{\vee}) 
\end{equation}
 given by multiplication with the ``twisted'' coevaluation $s_{M^{\vee},M} \circ \coev$. To do this, it suffices to construct a system of morphisms $\tau_n \colon \Sym^n_A(M) \ten{A} \Sym^n_A(M^{\vee}) \rightarrow A$ such that $\tau_0=\id_A$ and the triangle
\begin{equation}\label{eqn:tau_triangle}
\vcenter{
\xymatrix{ \Sym^n_A(M) \ten{A} \Sym^n_A(M^{\vee}) \ar[rd]_{\tau_n} \ar[rr]^{\delta} && \Sym^{n+1}_A(M) \ten{A} \Sym^{n+1}_A(M^{\vee}) \ar[ld]^{\tau_{n+1}} \\ & A}
}
\end{equation}
 is commutative for all $n \in \mathbb{N}$.

 The basic idea of the construction is simple. The morphism $\delta$ is given by a coevaluation and a symmetry, so it makes sense to try to construct $\tau$ by using evaluations. The composite of evaluation, symmetry, and coevalution is equal to the rank. By assumption, the rank of $M$ (and hence of $\Sym^n_A(M)$) is a non-zero natural number in $\mathbb{Q} \subseteq R$, so we can divide by it. This way we can hope to get the desired commutativity of the above triangle.
 
 More precisely, we let $\tau^{\prime}_n$ be the composite
 \[
 \xymatrix{\Sym^n_A(M) \ten{A} \Sym^n_A(M^{\vee}) \ar[r]^-{\sigma_n \otimes \sigma_n} & M^{\ten{A} n} \ten{A} (M^{\vee})^{\ten{A} n} \ar[r]^-{\ev} & A}
 \]
 where
 \[
 \xymatrix{ \Sym^n_A(M) \ar[r]^-{\sigma_n} & M^{\ten{A} n} \ar[r]^-{\pi_n} & \Sym^n_A(M)}
 \]
 denotes the splitting of the idempotent $s \slash n!$ (and similarly for $M^{\vee}$). The evaluation morphism of $M^{\ten{A} n}$ is built from $n$ copies of the evaluation morphism of $M$, and the morphism $\tau_n$ is given by $\tau^{\prime}_n \slash \rk(\Sym^n_A M)$. Note that this rank is indeed invertible since we have
  \[
  \rk(\Sym^n_A M)=\binom{\rk(M)+n-1}{n}
  \]
 by Lemma~\ref{lemma:trace_of_exterior_power}, which is a non-zero natural number by assumption on $M$.
 
 To check that Diagram~\eqref{eqn:tau_triangle} is commutative, we first note that
 \[
 \xymatrix{ M^{\ten{A} n} \ten{A} (M^{\vee})^{\ten{A} n} \ar[r] & M^{\ten{A} n+1} \ten{A} (M^{\vee})^{\ten{A} n+1} \ar[d]^{\pi_{n+1} \ten{A} \pi_{n+1}} \\ \Sym^n_A(M) \ten{A} \Sym^n_A(M^{\vee}) \ar[r]_{\delta} \ar[u]^{\sigma_n \ten{A} \sigma_n} &  \Sym^n_A(M) \ten{A} \Sym^n_A(M^{\vee})}
 \]
 is commutative, where the top morphism is given by left multiplication with the twisted coevaluation morphism
 \[
 \xymatrix{  A \ar[r]^-{\coev} & M^{\vee} \ten{A} M \ar[r]^-{s_{M^{\vee},M}} & M \ten{A} M^{\vee}}
 \]
 in the tensor product $T(M) \ten{A} T(M^{\vee})$ of the tensor algebra of $M$ and of $M^{\vee}$.  This is readily checked by precomposing with the epimorphism $\pi_n \ten{A} \pi_n$. 
 This reduces the problem to checking that the diagram
 \begin{equation}\label{eqn:section_diagram}
\vcenter{
 \xymatrix{ M^{\ten{A} n} \ten{A} (M^{\vee})^{\ten{A} n} \ar[r] & M^{\ten{A} n+1} \ten{A} (M^{\vee})^{\ten{A} n+1} \ar[d]^-{s \slash (n+1)! \ten{A} s \slash (n+1)!} \\ 
   \Sym^n_A(M) \ten{A} \Sym^n_A(M^{\vee}) \ar[d]_{\sigma_n \ten{A} \sigma_n} \ar[u]^{\sigma_n \ten{A} \sigma_n} &  M^{\ten{A} n+1} \ten{A} (M^{\vee})^{\ten{A} n+1}  \ar[d]^{\ev} \\
   M^{\ten{A} n} \ten{A} (M^{\vee})^{\ten{A} n} \ar[r]_-{\ev}  &  A }
}
 \end{equation}
 commutes up to a factor of
 \[
 \frac{\rk\bigl(\Sym^{n+1}_A (M)\bigr)}{ \rk\bigl(\Sym^{n}_A (M)\bigr)}=\frac{\rk (M) + n }{n+1}
 \]
 (in the sense that the top composite is equal to the above factor times the bottom composite).
  
 The fact that the symmetry $s_{M,M}$ is dual to $s_{M^{\vee},M^{\vee}}$ implies that
\[
 \ev \circ \id \ten{A} s_{M^{\vee},M^{\vee}} = \ev \circ s_{M,M} \ten{A} \id \smash{\rlap{,}} 
\]
 where $\ev$ denotes the evaluation morphism of $M \ten{A} M$. Applying this to the idempotent $s \slash (n+1)!$ on the $(n+1)$-fold tensor product of $M$, we find that the equation
 \begin{align*}
 \ev \circ \bigl( s \slash (n+1)! \ten{A} s \slash (n+1)! \bigr)
 &= \ev \circ \Bigl( \bigl( s \slash (n+1)! \bigr)^2 \ten{A} \id \Bigr) \\
 &= \ev \circ \bigl( s \slash (n+1)! \ten{A} \id \bigr) \\
 &= \frac{1}{(n+1)!} \sum_{\rho \in \Sigma_{n+1}} \ev \circ \rho \ten{A} \id 
 \end{align*}
 holds (where now $\ev$ denotes the evaluation morphism of the $(n+1)$-fold tensor product of $M$). Thus the top composite of Diagram~\eqref{eqn:section_diagram} decomposes into a corresponding sum indexed by $\rho \in \Sigma_{n+1}$. Each summand is given by the composite of the inclusion $\sigma_n \ten{A} \sigma_n \colon \Sym_A^n(M) \ten{A} \Sym_A^n(M^{\vee}) \rightarrow M^{\ten{A} n} \ten{A} (M^{\vee})^{\ten{A} n}$, followed by the morphism
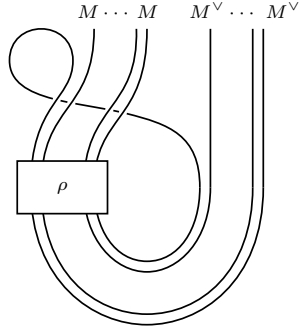
\begin{center} 
\begin{tikzpicture}[y=0.80pt,x=0.80pt,yscale=-1, inner sep=0pt, outer sep=0pt, every text node part/.style={font=\scriptsize} ]
  \path[draw=black,line join=miter,line cap=butt,line width=0.650pt]
    (65.0000,932.3624) .. controls (65.0000,977.3624) and (95.0000,997.3624) ..
    (120.0000,997.3624) .. controls (145.0001,997.3624) and (175.0001,977.3624) ..
    (175.0001,932.3624);
  \path[draw=black,line join=miter,line cap=butt,line width=0.650pt]
    (95.0000,932.3624) .. controls (95.0000,952.3624) and (105.0000,967.3624) ..
    (120.0000,967.3624) .. controls (135.0001,967.3624) and (145.0001,952.3624) ..
    (145.0001,932.3624);
  \path[draw=black,line join=miter,line cap=butt,line width=0.650pt]
    (150.0001,932.3624) .. controls (150.0001,912.3624) and (150.0001,882.3624) ..
    (150.0001,857.3624);
  \path[draw=black,line join=miter,line cap=butt,line width=0.650pt]
    (175.0001,932.3624) .. controls (175.0001,917.3624) and (175.0001,877.3624) ..
    (175.0001,857.3624);
  \path[draw=black,line join=miter,line cap=butt,line width=0.650pt]
    (70.0000,932.3624) .. controls (70.0000,972.3624) and (95.0000,992.3624) ..
    (120.0000,992.3624) .. controls (145.0001,992.3624) and (170.0001,972.3624) ..
    (170.0001,932.3624);
  \path[draw=black,line join=miter,line cap=butt,line width=0.650pt]
    (90.0000,932.3624) .. controls (90.0000,957.3624) and (105.0000,972.3624) ..
    (120.0000,972.3624) .. controls (135.0001,972.3624) and (150.0001,957.3624) ..
    (150.0001,932.3624);
  \path[draw=black,line join=miter,line cap=butt,line width=0.650pt]
    (170.0001,932.3624) .. controls (170.0001,907.3624) and (170.0001,882.3624) ..
    (170.0001,857.3624);
  \path[draw=black,line join=miter,line cap=butt,line width=0.650pt]
    (95.0000,932.3624) .. controls (95.0000,897.3624) and (120.0000,902.3624) ..
    (120.0000,857.3624)(90.0000,932.3624) .. controls (90.4537,892.0432) and
    (116.0440,896.6232) .. (115.0000,857.3624)(70.0000,932.3624) .. controls
    (70.0000,892.3624) and (95.0000,897.3624) ..
    (95.0000,857.3624)(65.0000,932.3624) .. controls (65.0000,892.3624) and
    (85.0000,890.8981) .. (85.0000,872.3624) .. controls (85.0000,852.3624) and
    (55.0000,852.3624) .. (55.0000,872.3624) .. controls (55.0000,882.5317) and
    (63.8273,887.3482) .. (75.4991,890.4419)(79.1406,891.3358) .. controls
    (80.2007,891.5776) and (81.2777,891.8093) ..
    (82.3678,892.0331)(86.1256,892.7685) .. controls (90.7998,893.6453) and
    (95.6239,894.4315) .. (100.3072,895.3036)(104.0868,896.0368) .. controls
    (104.9939,896.2211) and (105.8923,896.4106) ..
    (106.7798,896.6067)(110.4418,897.4730) .. controls (130.2205,902.4372) and
    (145.0001,907.5468) .. (145.0001,932.3624);
  \begin{scope}[shift={(493.3071,-873.98425)}]
    \path[fill=black] (-406.3071,1726.3464) node[above right] (text8483-6) {$M
      \cdots$       };
    \path[fill=black] (-378.3071,1726.3464) node[above right] (text8483-6-8) {$M$
      };
    \path[fill=black] (-317.3071,1726.3464) node[above right] (text8483-6-8-3)
      {$M^{\vee}$       };
    \path[fill=black] (-353.3071,1726.3464) node[above right] (text8483-6-8-3-2)
      {$M^{\vee} \cdots$       };
    \path[fill=black] (-413.3071,1806.3464) node[rectangle, draw, line width=0.65pt,
      minimum width=12mm, minimum height=7mm, fill=white, inner sep=0.25mm] (text8483-6-8-3-1) {$\rho$
      };
  \end{scope}

\end{tikzpicture}

\end{center}
 in the string diagram notation of \cite{JOYAL_STREET_TENSOR} (with function composition read from top to bottom). Here the cap denotes the coevaluation and the cups denote the evaluation morphism of $M$. 
 
 There are now two cases. If $\rho(1)=n+1$, then this composite is equal to $\rk(M) \cdot \tau^{\prime}$. Indeed, in this case the twisted coevaluation and one of the evaluations match up, and we do get a copy of $\rk(M)$. Morever, the composite of the section $\sigma_n$ with any morphism built out of symmetries on $M^{\ten{A} n}$ is equal to $\sigma_n$ by definition of the idempotent $s \slash n!$, hence the summand is equal to $\rk(M)\cdot \ev \circ \sigma_n \ten{A} \sigma_n = \rk(M) \cdot \tau_A^{\prime}$, as claimed.
 
 If $\rho(1) \neq n+1$, then the triangular identities imply that the twisted coevaluation will cancel with one of the evaluation morphisms of $M$. Thus in the second case the summand is simply equal to $\tau_n^{\prime}$. It therefore suffices to count the number of times these two cases occur.
 
 There are exactly $n!$ permutations in $\Sigma_{n+1}$ such that $\rho(1)=n+1$. Thus the top composite of Diagram~\ref{eqn:section_diagram} is equal to
\begin{align*}
\frac{n!}{(n+1)!} \rk(M) \cdot \tau_n^{\prime} + \frac{(n+1)!-n!}{(n+1)!} \tau_n^{\prime} 
&=\biggl( \frac{\rk(M)}{n+1} + \frac{n}{n+1} \biggr) \tau^{\prime}_n \\
&=\frac{\rk(M)+n}{n+1} \tau^{\prime}_n \smash{\rlap{,}}
\end{align*}
 as claimed. It follows that the triangle of Diagram~\ref{eqn:tau_triangle} is indeed commutative, which shows that the morphisms $\tau_n$ assemble into the desired retraction of the unit $A \rightarrow B$ as a morphism of $A$-modules. 
\end{proof}

\section{Applications}

\subsection{Quasi-coherent sheaves on algebraic spaces}
 The notion of a \emph{neutral} Tannakian category gives a characterization of the categories of coherent sheaves on stacks at the ``group-like'' end of the spectrum of stacks, that is, of stacks of the form $\ast \slash G$ for flat affine group schemes $G$. The fact that $\QCoh_{\fp}(-)$ sends finite bilimits to finite bicolimits allows us to give the following characterization of categories of coherent sheaves of algebraic spaces.

\begin{dfn}\label{def:essentially_discrete}
 An object $A$ of a 2-category $\ca{K}$ is called \emph{essentially discrete} (respectively \emph{essentially codiscrete}) if the category $\ca{K}(K,A)$ (respectively the category $\ca{K}(A,K)$) is equivalent to a discrete category for all objects $K \in \ca{K}$.
\end{dfn}

\begin{example}
 A stack on the $\fpqc$-site $\Aff_R$ is essentially discrete if and only if it is equivalent to a sheaf of sets. An algebraic stack in the sense of Goerss and Hopkins is essentially discrete if and only if it is equivalent to an \emph{algebraic space}, by which we mean a sheaf associated to an internal equivalence relation
\[
 \xymatrix{R \ar@<0.5ex>[r] \ar@<-0.5ex>[r] & U}
\]
 in the $\fpqc$-site $\Aff_R$ with the property that $R$ is flat over $U$. Note that this is not the usual definition of an algebraic space (it includes the assumption that the space is quasi-compact, but relaxes the requirement that the maps in the equivalence relation are \'etale).
\end{example}

\begin{prop}\label{prop:space_like_codiscrete}
 Let $\ca{A}$ be a weakly Tannakian category. Then $\ca{A} \simeq \QCoh_{\fp}(X)$ for an algebraic space $X$ if and only if $\ca{A}$ is an essentially codiscrete object in the 2-category $\ca{RM}$ of right exact symmetric monoidal categories. In other words, if and only if every right exact symmetric strong monoidal functor with domain $\ca{A}$ and codomain a right exact symmetric monoidal category has no non-identity endomorphisms.
\end{prop}

\begin{proof}
 From the recognition theorem we know that $\ca{A} \simeq \QCoh_{\fp}(X)$ for an Adams stack $X$. Since the 2-category $\ca{AS}$ of Adams stacks has finite bilimits, the power (or cotensor) $X^{\Aut}$ of $X$ by the free automorphism $\Aut \in \Gpd$ exists. Note that $X^{\Aut}$ is often called the inertia stack of $X$. To give a 1-cell with target $X^{\Aut}$ amounts to giving a 1-cell with target $X$ together with an automorphism. Thus $X$ is essentially discrete if and only if the canonical morphism $X \rightarrow X^{\Aut}$ is an equivalence. Note that this provides a characterization of essentially discrete objects in any 2-category with powers.

 The claim therefore follows from the facts that $\QCoh_{\fp} \colon \ca{AS}^{\op} \rightarrow \ca{RM}$ sends finite bilimits to finite bicolimits and reflects equivalences.
\end{proof}

 As one of the referees pointed out, it would be interesting to have an intrinsic characterization of codiscrete symmetric monoidal categories. Moreover, it would be very interesting to find a criterion for when the algebraic space $X$ in Proposition~\ref{prop:space_like_codiscrete} is a scheme. 

\subsection{Finite \'etale morphisms}\label{section:etale}
 A number of properties of morphisms between algebraic stacks are defined by using pullbacks and corresponding properties of morphisms of schemes. We use the example of finite\footnote{Since we are working with stacks that are not necessarily locally noetherian, we require that finite morphisms are locally given by algebras whose underlying module is finitely presentable (not just finitely generated).} \'etale morphisms to illustrate how Theorem~\ref{thm:pushouts} can be used to identify some of the duals of these properties in the world of tensor categories. Because of Proposition~\ref{prop:pushout_along_affine}, this works particularly well for classes of affine morphisms.

\begin{dfn}\label{def:separable_algebra}
 Let $\ca{C}$ be a cosmos\footnote{A \emph{cosmos} is a complete and cocomplete symmetric monoidal closed category.}. An algebra $A \in \ca{C}$ is called \emph{separable} if $A$, considered as an $A$-$A$-bimodule, is \emph{small projective}, in the sense that the covariant ($\ca{C}$-enriched) functor from $A$-$A$-bimodules to $\ca{C}$ represented by $A$ preserves all weighted colimits. In other words, if $A$ lies in the Cauchy completion of $A \otimes A^{\op}$.
\end{dfn}

% Note that in general, this notion of separability is weaker than demanding that the multiplication morphism is split as a morphism of $A$-$A$-bimodules. The reason for this is that an small projective object in the $\ca{C}$-enriched category of bimodules need not be projective in the ordinary sense in the underlying unenriched category unless the unit object of $\ca{C}$ is projective. In the geometric context, this happens if and only if the global sections functor is exact.

\begin{lemma}
 Let $\ca{C}$ be a cosmos. If $A \in \ca{C}$ is a commutative algebra, then $A$ is separable if and only if $A$ has a dual in the symmetric monoidal category $\ca{C}_{A\otimes A}$ of $A\otimes A$-modules (with tensor product $- \ten{A\otimes A}-$).
\end{lemma}

\begin{proof}
 See for example \cite[Proposition~8.4.1]{SCHAEPPI}, applied to the case $\ca{V}=\ca{C}$ and $\ca{B}$ the 1-object $\ca{C}$-category with endomorphism algebra $A \otimes A$.
\end{proof}

\begin{rmk}
 Note that separable algebras are often defined by asking that the multiplication $A \otimes A^{\op} \rightarrow A$ is split as a morphism of $A\otimes A^{\op}$-modules. However, if the unit $I \in \ca{C}$ is not a projective object (which is often the case for categories of quasi-coherent sheaves), then this is a stronger condition than the requirement of Definition~\ref{def:separable_algebra}.
\end{rmk}

 \begin{dfn}\label{dfn:separable_projective}
 Let $\ca{C}$ be a cosmos. A commutative algebra $A \in \ca{C}$ is called \emph{projective separable algebra} if $A$ is separable and the underlying object $A \in \ca{C}$ has a dual.
 \end{dfn}

\begin{thm}\label{thm:etale_characterization}
 Let $F \colon \ca{C} \rightarrow \ca{D}$ be a tensor functor between geometric categories. Then $F$ is induced by a finite \'etale morphism of algebraic stacks if and only if the right adjoint $U$ of $F$ is faithful and exact (that is, $F$ is cohomologically affine) and the commutative algebra $U(I)$ is projective separable (where $I \in \ca{D}$ denotes the unit object).

 Moreover, the 2-category of finite \'etale morphisms over a fixed Adams stack $X$ is equivalent to the opposite of the category of projective separable algebras in $\QCoh(X)$.
\end{thm}

 In order to prove this we use faithfully flat descent for objects with duals in pre-geometric categories. As one of the referees pointed out, the assumption that the categories are pre-geometric is not really necessary, but it allows us to give a simple argument using internal hom-objects instead of having to construct descent data explicitly.

\begin{prop}\label{prop:descent_for_duals}
 Let $\ca{C}$ be a pre-geometric category, and let $B \in \ca{C}$ be a faithfully flat algebra. Then $M \in \ca{C}$ is finitely presentable if and only if $B \otimes M \in \ca{C}_B$ is finitely presentable, and $M \in \ca{C}$ has a dual if and only if $B \otimes M \in \ca{C}_B$ has a dual.
\end{prop}

\begin{proof}
 The functor $B\otimes - \colon \ca{C}\rightarrow \ca{C}_B$ is a tensor functor, so it preserves duals. Its right adjoint is cocontinuous, hence finitary, so $B\otimes -$ preserves finitely presentable objects. This shows one direction.

 To see the converse, note that $B\otimes B \otimes M$ is finitely presentable as a $B\otimes B$-module if $B\otimes M$ is (by the above argument). Let $N$ be the filtered colimit of the $N_i$, and let $f \colon M \rightarrow N$. We need to show that $f$ factors through one of the $N_i$, and that any two such factorizations become equal after passage to a higher index.

 The morphisms $B\otimes f$ and $B\otimes B \otimes f$ factor through $B\otimes N_i$ and $B \otimes B \otimes N_j$ respectively. Passing to a higher index if necessary, we obtain factorizations for which the diagram
\[
 \xymatrix@!C=80pt{
B \otimes M \ar@<0.5ex>[r]^-{\eta \otimes B \otimes M} \ar[d]  \ar@<-0.5ex>[r]_-{B \otimes \eta \otimes M} & B \otimes B \otimes M \ar[d] \\
B \otimes N_i  \ar@<0.5ex>[r]^-{\eta \otimes B \otimes N_i} \ar@<-0.5ex>[r]_-{B \otimes \eta \otimes N_i} & B \otimes B \otimes N_i 
}
\]
 is serially commutative. Since
\[
 \xymatrix@C=50pt{M \ar[r]^-{\eta \otimes M} & B \otimes M \ar@<0.5ex>[r]^-{\eta \otimes B \otimes M} \ar@<-0.5ex>[r]_-{B \otimes \eta \otimes M} & B \otimes B \otimes M}
\]
 is an equalizer diagram for all $M \in \ca{C}$, we conclude that $f$ factors through $N_i$. To see that two such factorizations become equal after passage to a higher index we use the fact that $\eta \otimes N$ is monic.

 It remains to check that $M$ has a dual if $B\otimes M$ does. Here we use the fact that $\ca{C}$ is pre-geometric. Note that $\ca{C}_B$ is pre-geometric as well, so any object with dual is finitely presentable. The above argument implies that $M$ is finitely presentable as well. Therefore we can find an exact sequence
\[
 \xymatrix{D \ar[r] & D^{\prime} \ar[r] & M \ar[r] & 0}
\]
 in $\ca{C}$ where $D$ and $D^{\prime}$ have duals. It follows that for any $N \in \ca{C}$, the sequence
\[
 \xymatrix{0 \ar[r] & [M,N] \ar[r] & [D^{\prime},N]\cong (D^{\prime})^{\vee}\otimes N \ar[r] & [D,N]\cong D^{\vee} \otimes N}
\] 
 is exact, where $[-,-]$ denotes the internal hom functor. Left exactness of the tensor functor $B\otimes -$ therefore implies that the internal hom-object $[M,N]$ is preserved by $B\otimes -$. The claim now follows from the fact that $M$ has a dual if and only if the canonical morphism
\[
 [M,I]\otimes N \rightarrow [M,N]
\]
 is an isomorphism for all $N\in \ca{C}$.
\end{proof}

\begin{cor}\label{cor:separable_algebra_descent}
 Let $\ca{C}$ be a pre-geometric category, and let $B \in \ca{C}$ be a faithfully flat algebra. Then a commutative algebra $A \in \ca{C}$ is projective separable if and only if $B\otimes A \in \ca{C}_B$ is projective separable.
\end{cor}

\begin{proof}
 By Proposition~\ref{prop:descent_for_duals}, the underlying object of $A$ has a dual in $\ca{C}$ if and only if the underlying object of $B\otimes A$ has a dual in $\ca{C}_B$. The same proposition, applied to the pre-geometric category $\ca{C}_{A\otimes A}$ and the faithfully flat $A\otimes A$-algebra $B\otimes A\otimes A \cong (B\otimes A) \ten{B} (B\otimes A)$ shows that $A$ is a separable algebra in $\ca{C}$ if and only if $B\otimes A$ is a separable algebra in $\ca{C}_B$.
\end{proof}

\begin{proof}[Proof of Theorem~\ref{thm:etale_characterization}.]
 Let $f \colon X \rightarrow Y$ be a finite \'etale morphism, and let $Y_0 \rightarrow Y$ be a faithfully flat covering by an affine scheme. Let $B \in \QCoh(Y)$ be the faithfully flat affine algebra corresponding to $Y_0 \rightarrow Y$. Since $f$ is affine, the functor $f^{\ast}$ is equivalent to tensoring with a commutative algebra $A \in \QCoh(Y)$. We need to show that $A$ is a projective separable algebra.

 To see this, note that the pullback of $f$ along the covering $Y_0 \rightarrow Y$ is a finite \'etale morphism between affine schemes, so it is given by tensoring with a separable projective algebra\footnote{In the case of affine schemes, the two definitions of separable projective algebras agree.} (see \cite[Propositions~6.9 and 6.11]{LENSTRA}). On the other hand, the inverse image functor of this pullback is given by $(B \otimes A) \ten{B} -$ (see Proposition~\ref{prop:pushout_along_affine} and Theorem~\ref{thm:pushouts}). By faithfully flat descent it follows that $A$ is indeed projective separable (see Corollary~\ref{cor:separable_algebra_descent}).

 Conversely, assume that $f^{\ast}$ is equivalent to tensoring with a projective separable algebra. Then the pullback of $f$ along any morphism $g\colon Z \rightarrow Y$ has the same property. Indeed, by Proposition~\ref{prop:pushout_along_affine} and Theorem~\ref{thm:pushouts}, the inverse image functor of the pullback is given by tensoring with the commutative algebra $g^{\ast} A$, and tensor functors preserve separable projective algebras. This applies in particular in the case where $Z$ is affine, hence $f$ is finite \'etale by \cite[Propositions~6.9 and 6.11]{LENSTRA}.

 The claim about the 2-category of finite \'etale morphisms over a fixed Adams stack $X$ follows from Lemma~\ref{lemma:module_category_universal_property}.
\end{proof}

\subsection{Infinite limits of Adams stacks in characteristic zero}\label{section:complete}

 We conclude with an application of the intrinsic description of geometric categories for the case that $R$ is a $\mathbb{Q}$-algebra given in Theorem~\ref{thm:description}. The fact that $\QCoh_{\fp}(-)$ preserves finite limits raises the question if the same is true for infinite limits. This would follow if we could show that both the 2-category of Adams stacks is closed under infinite limits in the 2-category of stacks on the $\fpqc$-site $\Aff_R$, and the 2-category of weakly Tannakian categories is closed under infinite colimits in the 2-category of right exact symmetric monoidal categories. 
 
  Here we only prove the latter, using Theorem~\ref{thm:description}. Note that this does imply that the 2-category of Adams stacks has infinite (bicategorical) limits, but it does not imply that the inclusion of Adams stacks in the 2-category of $\fpqc$-stacks preserves these infinite limits. It would be interesting to know if this is true, that is, if Adams stacks are closed under the formation of infinite products.

 \begin{thm}\label{thm:closure_under_colimits}
 Suppose that $R$ is a $\mathbb{Q}$-algebra. Then weakly Tannakian categories are closed under arbitrary bicategorical colimits in the 2-category $\ca{RM}$ of right exact symmetric monoidal $R$-linear categories.
 \end{thm}

 The following lemma can be checked directly by explicitly computing filtered bicolimits of right exact symmetric monoidal categories. Since this is a bit tedious, we will give a proof that uses some 2-categorical machinery instead. 

\begin{lemma}\label{lemma:filtered_bicolimits_in_rm}
 Filtered bicolimits in the category $\ca{RM}$ of right exact symmetric monoidal $R$-linear categories are computed as in the category of small categories. In particular, the forgetful functor
\[
 \ca{RM} \rightarrow \Cat
\]
 preserves filtered bicolimits.
\end{lemma}

\begin{proof}
 First note that the 2-category $\ca{RM}$ is equivalent to the slice of $\Mod_R^{\fp}$ over the 2-category of finitely cocomplete symmetric monoidal categories $\ca{A}$ with a zero object. Indeed, a symmetric monoidal right exact functor $\Mod^{\fp}_R \rightarrow \ca{A}$ in particular gives a morphism $R \rightarrow \ca{A}(I,I)$ of rings, and its existence implies that the binary coproduct of the unit with itself is a biproduct. Therefore such a category $\ca{A}$ is indeed $R$-linear in a unique way compatible with the monoidal structure.

 This reduces the problem to computing filtered bicolimits in the category of symmetric monoidal categories with finite colimits and zero object. Since the Kelly tensor product $-\boxtimes -$ of finitely cocomplete pointed categories is closed, it preserves filtered colimits in each variable. Using the fact that the diagonal of a filtered colimit is bifinal (this is readily checked using \cite[Proposition~A.2.2]{SCHAEPPI_STACKS}), it follows that the pseudomonad for finitely cocomplete symmetric monoidal pointed categories commutes with filtered bicolimits.

 From \cite[Theorems~2.4 and 2.5]{DUBUC_STREET} it follows that filtered bicolimits of categories with finite colimits and zero object (and functors between them that preserve finite colimits and zero objects) are computed as in the category of small categories.
\end{proof}

 The forgetful functor $\ca{RM} \rightarrow \Cat$ has a left biadjoint $F$. We call a right exact tensor category \emph{finitely presentable} if it is equivalent to a finite bicolimit of tensor categories of the form $F\ca{A}$ for a finite category $\ca{A}$.

\begin{example}\label{example:free_on_dual}
 The free right exact symmetric monoidal category on an object with a dual is finitely presentable. Indeed, this category can be obtained by taking the free right exact symmetric monoidal on two objects $A$ and $B$, then freely gluing in two morphisms $I \rightarrow B \otimes A$ and $A\otimes B \rightarrow I$, and then imposing the two triangle identities.
\end{example}

\begin{lemma}\label{lemma:fp_means_fp}
 If $\ca{A}$ is a finitely presentable right exact tensor category, then the 2-functor $\ca{RM}(\ca{A},-)$ preserves filtered bicolimits.
\end{lemma}

\begin{proof}
 Since filtered bicolimits commute with finite bilimits in $\Cat$, we can reduce to the case where $\ca{A}$ is free on a finite category. The claim therefore follows by adjunction from Lemma~\ref{lemma:filtered_bicolimits_in_rm}.
\end{proof}

\begin{proof}[Proof of Theorem~\ref{thm:closure_under_colimits}]
 Since we already know that weakly Tannakian categories are closed under finite bicolimits, it suffices to show closure under filtered bicolimits. In other words, we need to check that the filtered colimit $\ca{A} \in \ca{RM}$ of weakly Tannakian categories $\ca{A}$ satisfies the conditions of Theorem~\ref{thm:intrinsic}. 

 Any object $A\in\ca{A}$ is isomorphic to an object in the image of one of the right exact symmetric strong monoidal $R$-linear functors $\ca{A}_i \rightarrow \ca{A}$. Thus there exists an object with dual $D$ and an epimorphism $D \rightarrow A$ (because the same is true in $\ca{A}_i$), so Condition~(i) holds.

 To see that the other conditions hold, note that any object with a dual in $\ca{A}$ is isomorphic to an image of an object \emph{with dual} under one of the functors $\ca{A}_i \rightarrow \ca{A}$. This follows from Lemma~\ref{lemma:fp_means_fp}, applied to the free right exact symmetric monoidal category on an object with a dual (which is finitely presentable by Example~\ref{example:free_on_dual}).  
 
 Since large exterior powers of objects with duals in $\ca{A}_j$ are zero, the same must be true in $\ca{A}$, so $\ca{A}$ satisfies Condition~(ii). Similarly, if for some object with dual $A \in \ca{A}$ the rank of the image in $\ca{A}$ is equal to zero, the same must be true in some $\ca{A}_j$. Since $\ca{A}_j$ is weakly Tannakian, this implies that the image of $A$ in $\ca{A}$ is zero. This shows that Condition~(iii) holds.

 To see that Condition~(iv) holds, note that the property of being an epimorphism can be checked via the equation asserting that the identity and the zero morphism of the cokernel are equal. This allows us to build a free right exact symmetric monoidal on an epimorphism as a finite bicolimit of the free right exact symmetric monoidal category on a single morphism.

 From Lemma~\ref{lemma:fp_means_fp} it follows that an epimorphism between objects with duals in $\ca{A}$  is isomorphic to the image of an epimorphism between duals under one of the functors $\ca{A}_i \rightarrow \ca{A}$. So Condition~(iv) must hold in $\ca{A}$ since it holds in $\ca{A}_i$.

 Finally, we need to check that $\ca{A}$ is ind-abelian. By \cite[Definition~1.1]{SCHAEPPI_TENSOR}, this amounts to checking that two conditions for finite diagrams, expressed using only cokernels and zero morphisms, hold in $\ca{A}$ if they hold in all the $\ca{A}_i$. This again follows readily from Lemma~\ref{lemma:fp_means_fp}.
\end{proof}

\bibliographystyle{amsalpha}
\bibliography{fiber}

\end{document}